\documentclass[12pt]{article}

\usepackage{amsmath} 
\usepackage{amsthm} 
\usepackage{amsfonts} 
\usepackage{amssymb} 
\usepackage{stmaryrd} 
\usepackage{mathtools} 

\usepackage{bm} 
\usepackage{bbm} 

\usepackage{array} 
\usepackage[inline]{enumitem} 
\usepackage{microtype}
\usepackage{caption}

\usepackage{tikz} 
\usetikzlibrary{calc} 
\usetikzlibrary{arrows} 
\usetikzlibrary{cd} 
\usetikzlibrary{decorations.markings} 

\usepackage[%
  pdftex,
  plainpages=false,
  pdfpagelabels,
  colorlinks,
  citecolor=black,
  linkcolor=black,
  urlcolor=black,
  filecolor=black,
  bookmarksopen=false
]{hyperref} 
\usepackage[all]{hypcap} 

\newtheoremstyle{thmstyle}
  {\medskipamount}
  {\smallskipamount}
  {\slshape}
  {0pt}
  {\bfseries}
  {.}
  { }
  {\thmname{#1}\thmnumber{ #2}{\normalfont\thmnote{ (#3)}}}

\newtheoremstyle{plainstyle}
  {\medskipamount}
  {\smallskipamount}
  {\rmfamily}
  {0pt}
  {\bfseries}
  {.}
  { }
  {\thmname{#1}\thmnumber{ #2}{\normalfont\thmnote{ (#3)}}}

\theoremstyle{thmstyle}
\newtheorem{theorem}{Theorem}[section]
\newtheorem{lemma}[theorem]{Lemma}

\newtheorem{proposition}[theorem]{Proposition}

\newtheorem{claim}[theorem]{Claim}

\theoremstyle{plainstyle}
\newtheorem{definition}[theorem]{Definition}
\newtheorem{remark}{Remark}

\newtheorem{example}{Example}

\newenvironment{proofof}[1]{\begin{proof}[Proof of #1.]}{\end{proof}}

\setlist[enumerate]{label={\roman*.}, ref={(\roman*)}}

\newcommand{\rn}{\bm}
\newcommand{\df}{\stackrel{\text{def}}{=}}

\newcommand{\disjcup}{\mathbin{\stackrel\cdot\cup}}

\newcommand{\comp}{\mathbin{\circ}}
\newcommand{\rest}{\mathord{\vert}}
\newcommand{\floor}[1]{\left\lfloor #1\right\rfloor}
\newcommand{\ceil}[1]{\left\lceil #1\right\rceil}
\let\:\colon

\def\up{\mathord{\uparrow}}
\def\down{\mathord{\downarrow}}

\def\tot{\leftrightarrow}

\DeclareMathOperator{\thk}{th}
\def\Dopen{D_\opertorname{open}}
\DeclareMathOperator{\Tr}{Tr}

\def\Dopen{D_{\operatorname{open}}}
\DeclareMathOperator{\Forb}{Forb}
\DeclareMathOperator{\im}{im}

\DeclareMathOperator{\id}{id}

\newcommand{\EE}{\mathbb{E}}
\newcommand{\NN}{\mathbb{N}}

\newcommand{\cC}{\mathcal{C}}
\newcommand{\cF}{\mathcal{F}}
\newcommand{\cL}{\mathcal{L}}
\newcommand{\cM}{\mathcal{M}}
\newcommand{\cP}{\mathcal{P}}
\newcommand{\cS}{\mathcal{S}}
\newcommand{\cT}{\mathcal{T}}
\newcommand{\cU}{\mathcal{U}}

\newcommand{\TLinOrder}{T_{\operatorname{LinOrder}}}
\newcommand{\TGraph}{T_{\operatorname{Graph}}}
\newcommand{\TkHypergraph}[1][k]{T_{#1\operatorname{-Hypergraph}}}
\newcommand{\TCycOrder}{T_{\operatorname{CycOrder}}}
\newcommand{\TTournament}{T_{\operatorname{Tournament}}}

\title{On the abstract chromatic number and its computability for finitely axiomatizable theories}
\author{Leonardo N.~Coregliano}

\begin{document}
\maketitle

\begin{abstract}
  The celebrated Erd\H{o}s--Stone--Simonovits theorem characterizes the asymptotic maximum edge density in
  $\cF$-free graphs as $1 - 1/(\chi(\cF)-1) + o(1)$, where $\chi(\cF)$ is the minimum chromatic number of a
  graph in $\cF$. In~\cite[Examples~25 and~31]{CR19}, it was shown that this result can be extended to the
  general setting of graphs with extra structure: the maximum asymptotic density of a graph with extra
  structure without some \emph{induced} subgraphs is $1 - 1/(\chi(I) - 1) + o(1)$ for an appropriately defined
  \emph{abstract chromatic number} $\chi(I)$. As the name suggests, the original formula for the abstract
  chromatic number is so abstract that its (algorithmic) computability was left open.

  In this paper, we both extend this result to characterize maximum asymptotic density of
  $t$-cliques in of graphs with extra structure without some \emph{induced} subgraphs in terms of $\chi(I)$ and we
  present a more concrete formula for $\chi(I)$ that allows us to show its computability when both the extra
  structure and the forbidden subgraphs can be described by a \emph{finitely} axiomatizable universal first-order
  theory. Our alternative formula for $\chi(I)$ makes use of a partite version of Ramsey's Theorem for
  structures on first-order relational languages.
\end{abstract}

\section{Introduction}

Two of the most famous theorems in extremal graph theory are Tur\'{a}n's Theorem~\cite{Tur41} characterizing
the maximum number of edges in a graph without $\ell$-cliques $K_\ell$ and Ramsey's Theorem~\cite{Ram29} that
says that for every $\ell$, a large enough $k$-uniform hypergraph must either contain an $\ell$-clique
$K^{(k)}_\ell$ or an $\ell$-independent set $\overline{K}^{(k)}_\ell$. The celebrated
Erd\H{o}s--Stone--Simonovits Theorem~\cite{ES46,ES66} generalizes Tur\'{a}n's Theorem by characterizing the
maximum asymptotic edge density when we instead forbid a family $\cF$ of non-induced subgraphs in terms of the
smallest chromatic number of a graph in $\cF$. In another direction, Erd\H{o}s~\cite{Erd62} generalized
Tur\'{a}n's Theorem by characterizing the maximum number of $t$-cliques $K_t$ in a graph without
$\ell$-cliques $K_\ell$ ($t < \ell$) and Alon--Shikhelman~\cite{AS16} provided the following analogue of the
Erd\H{o}s--Stone--Simonovits Theorem.

\begin{theorem}[Erd\H{o}s--Stone--Simonovits~\cite{ES46,ES66}, Alon--Shikhelman~\cite{AS16}]\label{thm:ESSAS}
  Let $t\in\NN$ and let $\cF$ be a non-empty family of finite non-empty graphs. The maximum number of copies of
  $t$-cliques $K_t$ in a graph $G$ with $n$ vertices and without any non-induced copies of elements of $\cF$
  is
  \begin{align*}
    \prod_{i=1}^{t-1} \left(1 - \frac{j}{\chi(\cF)-1}\right)\binom{n}{t} + o(n^t),
  \end{align*}
  where $\chi(\cF)\df\min\{\chi(F) \mid F\in\cF\}$ is the minimum chromatic number of a graph in $\cF$.
\end{theorem}

A relatively new type of generalization of the Tur\'{a}n and Erd\H{o}s--Stone--Simonovits theorems is to study
maximization of the asymptotic edge density in graphs with extra structure while forbidding non-induced copies
of some family $\cF$. This has been done for ordered graphs~\cite{PT06}, cyclically ordered
graphs~\cite{BKV03} and edge-ordered graphs~\cite{GMNPTV19} and in all these cases a theorem similar to
Theorem~\ref{thm:ESSAS} for $t=2$ is proved in terms of a suitable generalization of the chromatic number (see
also~\cite{Tar19} for a survey). However, all these cases were done in an ad hoc fashion.

A uniform and general treatment of this problem was first done in~\cite[Examples~25 and~31]{CR19}: in the
general case, we want to maximize asymptotic edge density in a hereditary family of graphs with some extra
structure. Note that even when restricted to the usual graphs without extra structure, this is already a
generalization of Theorem~\ref{thm:ESSAS} as the forbidden subgraphs are \emph{induced}. This general setting
is formally captured by using open interpretations $I\:\TGraph\leadsto T$ (see Section~\ref{sec:prelim} for
precise definitions). Informally, we consider an arbitrary family $\cM$ of structures that is closed under
induced substructures and a construction $I$ that produces a graph $I(M)$ from an element $M$ of $\cM$ in a
local way in the sense that to decide whether $\{v,w\}$ is an edge of $I(M)$, it is enough to know only
information of $M\rest_{\{v,w\}}$; the problem then consists of maximizing the asymptotic edge density of
$I(M)$ over all possible choices of $M$ as the size of $M$ goes to infinity for a given fixed $I$. For
example, the aforementioned setting of (cyclically) ordered graphs are captured using the construction $I$
that simply ``forgets'' the (cyclic) order. A less trivial example of local combinatorial construction that is
captured in this framework is that of the graph of inversions of a permutation.

In~\cite[Example~31]{CR19}, it was shown that in this general setting a result analogous to
Theorem~\ref{thm:ESSAS} when $t = 2$ still holds for an appropriately defined \emph{abstract chromatic number}
$\chi(I)$. However, the formula for $\chi(I)$ presented in~\cite[Equation~(16)]{CR19}
(see~\eqref{eq:abstractchromatic} in Section~\ref{sec:prelim} below) is considerably abstract and it was left
open if $\chi(I)$ was (algorithmically) computable even when $T$ is assumed to be finitely axiomatizable.

In this paper, we both generalize the result of~\cite{CR19} to the case $t\geq 3$
(Theorem~\ref{thm:abstractchromatic}) and provide an alternative, more concrete formula for $\chi(I)$
(Theorem~\ref{thm:Ramseyabstractchromatic}). Such formula allows us to deduce that when $T$ is \emph{finitely}
axiomatizable, then $\chi(I)$ is (algorithmically) computable from a list of the axioms of $T$ and a
description of $I$ (Theorem~\ref{thm:computability}). Our alternative formula is based on a partite version of
Ramsey's Theorem (Theorem~\ref{thm:Ramsey}) for universal theories that informally says that given
$\ell,m\in\NN$, there exists $n\in\NN$ such that for every model $M$ and every partition of $M$ into $\ell$
parts all of size at least $n$ must have a ``uniform'' submodel on the same partition with all parts of size
$m$ (this version of Ramsey's Theorem for disjoint unions of theories of hypergraphs follows
from~\cite[Section~5]{GRS90} and the non-partite version, when $\ell=1$, for general theories follows from the
general Ramsey theory for systems of~\cite{NR89}; see Section~\ref{subsec:prelim:Ramsey} for more details). By
using these different formulas for $\chi(I)$, we can retrieve the results of~\cite{PT06,BKV03,GMNPTV19} on
ordered graphs, cyclically ordered graphs and edge-ordered graphs, respectively from the general theory (see
Section~\ref{sec:concrete}).

\bigskip

The paper is organized as follows. In Section~\ref{sec:prelim}, we establish notation needed to formally state
our results. In Section~\ref{sec:main}, we state our main results. In Section~\ref{sec:Turan}, we show
Theorem~\ref{thm:abstractchromatic}, which is the generalization of Theorem~\ref{thm:ESSAS} in terms of the
abstract chromatic number. In Section~\ref{sec:Ramsey}, we show the partite version of Ramsey's Theorem,
Theorem~\ref{thm:Ramsey}. In Section~\ref{sec:Ramseyabstractchromatic}, we prove
Theorem~\ref{thm:Ramseyabstractchromatic}, which provides an alternative formula for the abstract chromatic
number and we prove Theorem~\ref{thm:computability} on its computability when the theory is finitely
axiomatizable. In Section~\ref{sec:noninduced}, we provide other alternative formulas for the abstract
chromatic number in the ``non-induced'' setting. In Section~\ref{sec:concrete}, we illustrate how to use the
general theory to obtain easier formulas for some concrete theories. We conclude with some final remarks and
open problems in Section~\ref{sec:conclusion}.

\section{Preliminaries}
\label{sec:prelim}

Throughout the text, we let $\NN\df\{0,1,\ldots\}$ be the set of non-negative integers and let
$\NN_+\df\NN\setminus\{0\}$. We also let $[n]\df\{1,\ldots,n\}$ (and set $[0]\df\varnothing$). The usage of
the arrow $\rightarrowtail$ for a function will always presume the function to be injective. For a set $V$,
let $(V)_\ell$ be the set of all injective functions $\alpha\:[\ell]\rightarrowtail V$ and for one such
$\alpha$, we will use the notation $\alpha_i$ for $\alpha(i)$ when convenient. Let $2^V$ be the set of all
subsets of $V$ and $\binom{V}{\ell}\df\{A\subseteq V \mid \lvert A\rvert = \ell\}$.

\subsection{General combinatorial objects as models of a canonical theory}

We will be working in (a small fraction of) the framework of~\cite{CR19}, in which combinatorial objects
are encoded as models of a canonical theory. Each theory has an underlying finite first-order relational
language $\cL$, that is, the language $\cL$ is a finite set of predicate symbols $P\in\cL$, whose \emph{arity}
we denote by $k(P)\in\NN_+$. All languages will be assumed to be finite first-order relational languages. An
\emph{open formula} is any formula that does not contain any quantifiers and a \emph{universal formula} is a
formula of the form $\forall x_1\cdots\forall x_n, F(x_1,\ldots,x_n)$, where $F$ is an open formula. A theory
$T$ is \emph{universal} if all of its axioms are universal formulas. A universal theory $T$ is
\emph{canonical} if for every $P\in\cL$, the theory $T$ entails ($\vdash$) the formula
\begin{align}\label{eq:canonical}
  \forall\vec{x}, \left(\bigvee_{1\leq i < j\leq k(P)} x_i = x_j\right) \to \neg P(x_1,\ldots,x_{k(P)}).
\end{align}
A \emph{translation} of a language $\cL_1$ into a language $\cL_2$ is a mapping $I$ that takes every predicate
symbol $P(x_1,\ldots,x_k)\in\cL_1$ to an open formula $I(P)(x_1,\ldots,x_k)$ in $\cL_2$. The translation is
extended to open formulas by declaring that it commutes with logical connectives. An \emph{open
  interpretation} of a universal theory $T_1$ over a language $\cL_1$ in a universal theory $T_2$ over a
language $\cL_2$ is a translation $I$ of $\cL_1$ into $\cL_2$ such that for every axiom $\forall\vec{x},
F(\vec{x})$ of $T_1$, we have $T_2\vdash\forall\vec{x}, I(F)(\vec{x})$; we denote open interpretations as
$I\:T_1\leadsto T_2$. Since every universal theory is isomorphic to a canonical
theory~\cite[Theorem~2.3]{CR19}, all theories will be assumed to be canonical. We will also omit universal
quantifiers when stating axioms of universal theories.

Given a language $\cL$, we let $T_\cL$ be the \emph{pure canonical theory} over $\cL$, that is, the theory
whose axioms are~\eqref{eq:canonical} for each $P\in\cL$. Other important examples of canonical theories
include the theory of $k$-hypergraphs $\TkHypergraph$, whose language contains a single predicate $E$ of arity
$k(E)\df k$ and whose axioms are~\eqref{eq:canonical} for $P=E$ and
\begin{align*}
  \forall\vec{x}, (E(x_1,\ldots,x_k)\to E(x_{\sigma(1)},\ldots,x_{\sigma(k)})) & & (\sigma\in S_k);
\end{align*}
the theory of (simple) graphs $\TGraph\df\TkHypergraph[2]$, the theory of (strict) linear orders
$\TLinOrder$, whose language contains a single binary predicate $\prec$ and whose axioms are
\begin{gather*}
  \forall x, \neg(x \prec x);\\
  \forall\vec{x}, (x_1\neq x_2 \to (x_1\prec x_2 \tot \neg(x_2\prec x_1)));\\
  \forall\vec{x}, (x_1\prec x_2\land x_2\prec x_3 \to x_1\prec x_3);
\end{gather*}
and the theory of tournaments $\TTournament$, whose language contains a single binary predicate $E$ and has
the axioms
\begin{gather*}
  \forall x, \neg E(x,x);\\
  \forall x_1\forall x_2, (x_1\neq x_2 \to (E(x_1,x_2)\tot\neg E(x_2,x_1))).
\end{gather*}

Given two theories $T_1$ and $T_2$ over languages $\cL_1$ and $\cL_2$, respectively, we let $T_1\cup T_2$ be
their \emph{disjoint union}, that is, the theory in the disjoint union $\cL_1\disjcup\cL_2$ of the languages
whose axioms are those of $T_1$ (about symbols in $\cL_1$) and $T_2$ (about symbols in $\cL_2$). The two most
important types of open interpretations are the \emph{structure-erasing} interpretations, which are open
interpretations of the form $I\:T_1\leadsto T_1\cup T_2$ that act identically on the language of $T_1$ and
\emph{axiom-adding} interpretations, which are open interpretations of the form $I\:T_1\leadsto T_2$ when
$T_2$ is obtained from $T_1$ by adding axioms and $I$ acts identically on the language of $T_1$. In fact,
every open interpretation $I\:T_1\leadsto T_2$ is of the form $I = J\comp A\comp S$, where $J$ is an
isomorphism, $A$ is axiom-adding and $S$ is structure-erasing (see~\cite[Remark~2]{CR19}).

A structure $K$ on $\cL$ is called \emph{canonical} if it satisfies~\eqref{eq:canonical} for every $P\in\cL$
(equivalently, if $K$ is a model of $T_\cL$). In this case, we denote the universe of $K$ by $V(K)$ (and call
its elements \emph{vertices}), denote by $\lvert K\rvert\df\lvert V(K)\rvert$ its \emph{size} and for each
$P\in\cL$, we let $R_P(K)\df\{\alpha\in(V(K))_{k(P)} \mid K\vDash P(\alpha_1,\ldots,\alpha_{k(P)})\}$ be the
set of all (necessarily injective) tuples of vertices that satisfy $P$ in $K$. All our structures will be
assumed to be canonical unless explicitly mentioned otherwise. An \emph{embedding} of a structure $K_1$ in a
structure $K_2$ on $\cL$ is an injective function $f\:V(K_1)\rightarrowtail V(K_2)$ such that for every
$P\in\cL$ and every $\alpha\in(V(K_1))_{k(P)}$, we have
\begin{align*}
  \alpha\in R_P(K_1) & \tot f\comp\alpha \in R_P(K_2).
\end{align*}
An \emph{isomorphism} is an embedding that is also bijective and two structures $K_1$ and $K_2$ are
\emph{isomorphic} (denoted $K_1\cong K_2$) when there exists an isomorphism between them. For a structure $K$
on $\cL$, and a set $U\subseteq V(K)$, we let $K\rest_U$ be the \emph{substructure of $K$ induced by $U$},
that is, we have $V(K\rest_U)\df U$ and $R_P(K\rest_U)\df R_P(K)\cap (U)_{k(P)}$ for every $P\in\cL$.

As usual, a \emph{model} of a theory $T$ over $\cL$ is a structure on $\cL$ that satisfies all axioms of
$T$. For $n\in\NN$, we let $\cM_n[T]$ be the set of models on $n$ vertices \emph{up to isomorphism} and let
$\cM[T]\df\bigcup_{n\in\NN}\cM_n[T]$. We think of elements of $\cM_n[T]$ in terms of a representative model
$M$ with vertex set $V(M) = [n]$.

For a canonical structure $K\in\cM_n[T_\cL]$, the \emph{open diagram} $\Dopen(K)(x_1,\ldots,x_n)$ of $K$ is
the open formula
\begin{align*}
  \bigwedge_{1\leq i < j\leq n} x_i\neq x_j
  \land \bigwedge_{\substack{P\in\cL\\\alpha\in R_P(K)}} P(x_{\alpha_1},\ldots,x_{\alpha_{k(P)}})
  \land \bigwedge_{\substack{P\in\cL\\\alpha\in [n]^{k(P)}\setminus R_P(K)}} \neg P(x_{\alpha_1},\ldots,x_{\alpha_{k(P)}}).
\end{align*}

Given a family $\cF$ of models of a theory $T$, we let $\Forb_T(\cF)$ be the theory obtained from $T$ by
adding the axioms
\begin{align*}
  \forall \vec{x}, & \neg\Dopen(F)(\vec{x}) & (F\in\cF).
\end{align*}
Clearly, every canonical theory $T$ satisfies $T=\Forb_{T_\cL}(\cF)$ for $\cF\df\cM[T_\cL]\setminus\cM[T]$.

Given an open interpretation $I\:T_1\leadsto T_2$ and a model $M$ of
$T_2$, there is a naturally defined model $I(M)$ of $T_1$ given by $V(I(M))\df V(M)$ and
$R_P(I(M))\df\{\alpha\in(V(M))_{k(P)} \mid M\vDash I(P)(\alpha_1,\ldots,\alpha_{k(P)})\}$.

A canonical theory $T$ is called \emph{degenerate} if $\cM_n[T]=\varnothing$ for some $n\in\NN$ (equivalently,
if $T$ does not have an infinite model).

Given two models $M,N\in\cM[T]$ with $\lvert M\rvert\leq\lvert N\rvert$, the \emph{(unlabeled induced)
  density} of $M$ in $N$ is
\begin{align*}
  p(M,N)
  &\df
  \frac{
    \left\lvert\left\{V\in\binom{V(N)}{\lvert M\rvert} \;\middle\vert\;
    N\rest_V\cong M\right\}\right\rvert
  }{\binom{\lvert N\rvert}{\lvert M\rvert}},
\end{align*}
that is, it is the normalized number of submodels of $N$ that are isomorphic to $M$.

\subsection{The general Tur\'{a}n density and the abstract chromatic number}

In the theory of graphs $\TGraph$, we denote the \emph{complete graph} on $n$ vertices by
$K_n\in\cM_n[\TGraph]$, that is, we have $R_E(K_n)\df (V(K_n))_2$; we denote the \emph{empty graph} on $n$
vertices by $\overline{K}_n$, that is, we have $R_E(\overline{K}_n) \df \varnothing$; and we denote the
\emph{$\ell$-partite Tur\'{a}n graph} of size $n$ by $T_{n,\ell}\in\cM_n[\TGraph]$, that is, $T_{n,\ell}$ is
the complete $\ell$-partite graph with parts of sizes either $\floor{n/\ell}$ or $\ceil{n/\ell}$, or in a
formula, we have $R_E(T_{n,\ell})\df\{\alpha\in([n])_2 \mid \alpha_1\equiv\alpha_2\pmod{\ell}\}$. For graphs
$G$ and $H$, we write $G\subseteq H$ if $H$ has a \emph{non-induced copy} of $G$, that is, if there exists
$f\:V(G)\rightarrowtail V(H)$ that maps edges of $G$ to edges of $H$, or in formulas, for every $\alpha\in
R_E(G)$, we have $f\comp\alpha\in R_E(H)$. Recall that a proper coloring of a graph $G$ is a function
$f\:V(G)\to[\ell]$ such that $\forall\alpha\in R_E(G), f(\alpha_1)\neq f(\alpha_2)$ and the \emph{chromatic
  number} of $G$ is the minimum $\ell\in\NN$ such that there exists a proper coloring of $G$ of the form
$f\:V(G)\to[\ell]$.

\begin{definition}[Abstract Tur\'{a}n density]
  For an open interpretation $I\:\TGraph\leadsto T$ and $t\in\NN$, the \emph{$t$-Tur\'{a}n density} of $I$ is
  defined as
  \begin{align}\label{eq:Turandensity}
    \pi^t_I & \df \lim_{n\to\infty} \sup_{N\in\cM_n[T]} p(K_t,I(N)).
  \end{align}
\end{definition}

The existence of the limit in~\eqref{eq:Turandensity} follows from the fact that the sequence is non-increasing
(for $n\geq t$). This can be proved by the standard averaging argument of extremal combinatorics: if $T$ is
degenerate, then the sequence is eventually constant equal to $-\infty$; otherwise, if $N_0\in\cM_{n+1}[T]$
($n\geq t$) maximizes $p(K_t,I(N_0))$, then picking uniformly at random a subset $\rn{U}$ of $V(N)$ of size
$n$, we conclude that
\begin{align*}
  \sup_{N\in\cM_n[T]} p(K_t,I(N)) & \geq \EE[p(K_t,I(N_0\rest_{\rn{U}}))] = p(K_t,I(N_0))
  = \sup_{N\in\cM_{n+1}[T]} p(K_t,I(N)).
\end{align*}

Note also that since $\pi^t_I$ is stated in terms of densities, when we count copies of $K_t$
instead, we incur an $o(n^t)$ error.

\begin{definition}[Abstract chromatic number~\protect{\cite[Equation~(16)]{CR19}}]
  For an open interpretation $I\:\TGraph\leadsto T$, the \emph{abstract chromatic number of $I$} is defined
  as\footnote{The formula in~\eqref{eq:abstractchromatic} is actually a slight modification
    of~\cite[Equation~(16)]{CR19}, forcing $0$ to belong to the set. This is done so that we can also cover
    degenerate theories $T$.}
  \begin{align}\label{eq:abstractchromatic}
    \chi(I)
    & \df
    \sup\{\ell\in\NN_+ \mid \forall n\in\NN, \exists N\in\cM_n[T], T_{n,\ell}\subseteq I(N)\}\cup\{0\} + 1.
  \end{align}
\end{definition}

Note that $\chi(I)\in\NN_+\cup\{\infty\}$ because the set in~\eqref{eq:abstractchromatic} always contains
$0$. Furthermore, note that if $T$ is degenerate, then $\chi(I) = 1$ as the set
in~\eqref{eq:abstractchromatic} is $\{0\}$.

The usual Tur\'{a}n density studied in Theorem~\ref{thm:ESSAS} is $\pi^t_{I_\cF}$ for the axiom-adding
interpretation $I_\cF\:\TGraph\leadsto\Forb_{\TGraph}^+(\cF)$, where $\Forb_{\TGraph}^+(\cF)$ is the theory
obtained from $\TGraph$ by adding for each $F\in\cF$ the axiom
\begin{align*}
  \forall x_1\cdots\forall x_m,
  \neg\left(\bigwedge_{1\leq i < j\leq m} x_i\neq x_j\land\bigwedge_{\alpha\in R_E(F)} E(x_{\alpha_1},x_{\alpha_2})\right),
\end{align*}
where we rename the vertices of $F$ so that $V(F) = [m]$. We will see in Proposition~\ref{prop:usual} that in
this case $\chi(I_\cF)$ is equal to the usual chromatic number $\chi(\cF)\df\inf\{\chi(F) \mid F\in\cF\}$
except for when $\cF$ is empty or contains an empty graph; more precisely, we have $\chi(I_\cF) =
\max\{\chi(\cF),1\}$.

\subsection{Partite Ramsey numbers}
\label{subsec:prelim:Ramsey}

As we mentioned in the introduction, our alternative formula for the abstract chromatic number is based on a
partite version of Ramsey's Theorem for universal theories. The first step to this
version is identifying what are the ``uniform'' structures that are unavoidable in a large structure.
Let us start with the easier case in which all predicate symbols are symmetric: this is captured by the
theories of $\vec{k}$-hypergraphs defined below.

\begin{definition}[$\vec{k}$-hypergraphs]
  For $\vec{k} = (k_1,\ldots,k_t)\in\NN_+^t$, let $\TkHypergraph[\vec{k}]\df\bigcup_{i=1}^t
  \TkHypergraph[k_i]$, where we denote the $i$-th predicate symbol by $E_i$. A \emph{$\vec{k}$-hypergraph} is
  a model of $\TkHypergraph[\vec{k}]$. The \emph{$i$-th edge set} of a $\vec{k}$-hypergraph $H$ is the set
  $E_i(H)\df\{\im(\alpha) \mid \alpha\in R_{E_i}(H)\}$.
\end{definition}

Any ordered partition $(V_1,\ldots,V_\ell)$ of a set $V$ can be described alternatively by the function
$f\:V\to[\ell]$ such that $v\in V_{f(v)}$ for every $v\in V$. We can then classify the subsets $e\subseteq V$
according to how many points $e$ contains in each of the parts $V_i$. The notions of Ramsey patterns and
uniform $\vec{k}$-hypergraphs defined below explore this classification.

\begin{definition}[$\vec{k}$-hypergraph Ramsey patterns and uniform $\vec{k}$-hypergraphs]
  Recall that for $\ell,k\in\NN_+$, a \emph{weak composition} of $k$ of length $\ell$ is an $\ell$-tuple $q =
  (q_j)_{j=1}^\ell\in\NN^\ell$ such that $\sum_{j=1}^\ell q_j = k$. We denote the set of weak compositions of
  $k$ of length $\ell$ by $\cC_{\ell,k}$. 

  For $\ell\in\NN_+$, a \emph{$\vec{k}$-hypergraph $\ell$-Ramsey pattern} is a $t$-tuple $Q = (Q_i)_{i\in[t]}$
  such that $Q_i\subseteq\cC_{\ell,k_i}$ for every $i\in[t]$. We let $\cP_{\ell,\vec{k}}$ be the set of all
  $\vec{k}$-hypergraph $\ell$-Ramsey patterns.

  Given a $\vec{k}$-hypergraph $\ell$-Ramsey pattern $Q\in\cP_{\ell,\vec{k}}$, a $\vec{k}$-hypergraph $H$ and
  a function $f\:V(H)\to[\ell]$, we say that $H$ is \emph{$Q$-uniform with respect to $f$} if for every
  $i\in[t]$, the $E_i$-edges of $H$ are precisely those $e$ such that there exists some $q\in Q_i$ such that $e$
  contains exactly $q_j$ points in $f^{-1}(j)$, or in formulas we have
  \begin{align*}
    E_i(H)
    & =
    \left\{e\in\binom{V(H)}{k_i} \;\middle\vert\; (\lvert e\cap f^{-1}(j)\rvert)_{j\in[t]}\in Q_i\right\},
  \end{align*}
  which is in turn equivalent to
  \begin{align*}
    R_{E_i}(H) & = \{\alpha\in (V(H))_{k_i} \mid (\lvert(f\comp\alpha)^{-1}(j)\rvert)_{j\in[t]}\in Q_i\}.
  \end{align*}
\end{definition}

The partite version of Ramsey's Theorem for $\vec{k}$-hypergraphs (Theorem~\ref{thm:hypergraphRamsey} below)
says that uniform $\vec{k}$-hypergraphs cannot all be avoided as long as the parts of the partition are
sufficiently large.

\begin{definition}[Thickness and $\vec{k}$-hypergraph Ramsey numbers]
  The \emph{thickness} of a function $f\:V\to[\ell]$ is $\thk(f)\df\min\{\lvert f^{-1}(i)\rvert \mid
  i\in[\ell]\}$.

  Given $\ell\in\NN_+$ and $m\in\NN$, the \emph{$(\ell,\vec{k},m)$-Ramsey number} $R_{\ell,\vec{k}}(m)$ is
  defined as the least $n\in\NN$ such that for every $\vec{k}$-hypergraph $H$ and every $f\:V(H)\to[\ell]$
  with $\thk(f)\geq n$, there exists $Q\in\cP_{\ell,\vec{k}}$ and a set $W\subseteq V(H)$ such that
  $\thk(f\rest_W)\geq m$ and $H\rest_W$ is $Q$-uniform with respect to $f\rest_W$.
\end{definition}

\begin{theorem}\label{thm:hypergraphRamsey}
  For every $\ell\in\NN_+$, every $m\in\NN$ and every $\vec{k}\in\NN_+^t$, the $(\ell,\vec{k},m)$-Ramsey
  number $R_{\ell,\vec{k}}(m)$ is finite.
\end{theorem}

Theorem~\ref{thm:hypergraphRamsey} above can be obtained e.g.\ by repeatedly applying~\cite[Theorem~5 of
  Section~5]{GRS90}, but we provide a proof via a reduction to Ramsey's original theorem for hypergraphs in
Section~\ref{sec:Ramsey}.

For the case of general universal theories, we have an extra technicality: predicate symbols are not
necessarily symmetric. The correct way of addressing this issue is illustrated by the case of the theory of
tournaments $\TTournament$. The unavoidable ``uniform'' models here are the transitive tournaments $\Tr_n$
(with $R_E(\Tr_n)\df\{\alpha\in ([n])_2 \mid \alpha_1 < \alpha_2\}$): for every $k\in\NN$, every sufficiently
large tournament $M$ must contain a transitive tournament of size $k$ as a
subtournament~\cite{Ste59,EM64}. Another way of seeing a transitive tournament is that there is an underlying
order $\leq$ of its vertices such that we can decide whether $\alpha\in([n])_2$ is in $R_E(\Tr_n)$ based only
on the relative order of $\alpha_1$ and $\alpha_2$ with respect to $\leq$. In the $\ell$-partite case, the
role of the order $\leq$ is played by the $\ell$-split orders defined below, which are tuples $(f,\preceq)$
such that $f\:V\to[\ell]$ encodes an $\ell$-partition and $\preceq$ orders each of the parts of this
partition.

\begin{definition}[Split orders]
  For $\ell\in\NN_+$ and a set $V$, an \emph{$\ell$-split order} over $V$ is a pair $(f,{\preceq})$, where
  $f\:V\to[\ell]$ and $\preceq$ is a partial order on $V$ such that
  \begin{align*}
    \forall v,w\in V, (f(v) = f(w) & \tot v\preceq w\lor w\preceq v),
  \end{align*}
  that is, two elements of $V$ are comparable under $\preceq$ if and only if they have the same image under
  $f$. We let $\cS_{\ell,V}$ be the set of all $\ell$-split orders over $V$ and for $k\in\NN$, we use the
  shorthand $\cS_{\ell,k}\df\cS_{\ell,[k]}$.

  When $\ell=1$, we will typically omit $f$ from the notation as it must be the constant function; with this
  abuse, we will think of $\cS_{1,V}$ as the set of all total orders on $V$.

  For a partial order $\preceq$ on a set $V$ and an injective function $g\:W\rightarrowtail V$, we let
  $\preceq_g$ be the partial order on $W$ defined by
  \begin{align*}
    w_1 \preceq_g w_2 & \iff g(w_1)\preceq g(w_2).
  \end{align*}
  If $W\subseteq V$, then we let ${\preceq_W}\df{\preceq_{\iota_W}}$, where $\iota_W\:W\rightarrowtail V$ is
  the canonical injection, that is, $\preceq_W$ is just the restriction ${\preceq}\cap (W\times W)$ of
  $\preceq$ to $W$.
\end{definition}

Note that for $g\:W\to V$ and $h\:U\to W$ and for a partial order $\preceq$ on $V$, we have ${(\preceq_g)_h} =
{\preceq_{g\comp h}}$. Furthermore, if $(f,{\preceq})\in\cS_{\ell,V}$, then $(f\comp
g,{\preceq_g})\in\cS_{\ell,W}$. Finally, note that there are finitely many $\ell$-split orders over $[k]$.

Given an $\ell$-split order $(f,{\preceq})\in\cS_{\ell,V}$ over $V$, we can classify the tuples $\alpha\in
(V)_k$ according to $(f\comp\alpha,{\preceq_\alpha})$, that is, $f\comp\alpha$ captures the values of $f$ on
the image of $\alpha$ and $\preceq_\alpha$ captures the partial order induced by $\preceq$ on the image of
$\alpha$. Just as in the case of $\vec{k}$-hypergraphs, the notions of Ramsey patterns, uniform structures and
Ramsey numbers defined below explore this classification.

\begin{definition}[Ramsey patterns, uniform structures and Ramsey number]
  Fix $\ell\in\NN_+$ and a language $\cL$. An \emph{$\ell$-Ramsey pattern on $\cL$} is a function $Q$ that
  maps each predicate symbol $P\in\cL$ to a collection $Q_P\subseteq\cS_{\ell,k(P)}$ of $\ell$-split orders on
  $[k(P)]$. We let $\cP_{\ell,\cL}$ be the set of all $\ell$-Ramsey patterns on $\cL$.

  Given an $\ell$-Ramsey pattern $Q\in\cP_{\ell,\cL}$ on $\cL$, a canonical structure $M$ on $\cL$ and an
  $\ell$-split order $(f,{\preceq})\in\cS_{\ell,V(M)}$ on $V(M)$, we say that $M$ is \emph{$Q$-uniform with
    respect to $(f,{\preceq})$} if for every $P\in\cL$, we have
  \begin{align*}
    R_P(M) & = \{\alpha\in (V(M))_{k(P)} \mid (f\comp\alpha, {\preceq_\alpha})\in Q_P\}.
  \end{align*}

  For a canonical structure $M$ on $\cL$, the \emph{$\ell$-Ramsey uniformity set} of $M$ is the set
  $\cU_\ell(M)$ of all $\ell$-Ramsey patterns $Q\in\cP_{\ell,\cL}$ such that $M$ is $Q$-uniform with respect
  to some $(f,{\preceq})\in\cS_{\ell,V(M)}$. We extend this definition to a family $\cF$ of canonical
  structures as $\cU_\ell(\cF)\df\bigcup_{M\in\cF}\cU_\ell(M)$.

  Given a canonical theory $T$ over $\cL$ and $m\in\NN$, the \emph{$(\ell,T,m)$-Ramsey number} $R_{\ell,T}(m)$
  is defined as the least $n\in\NN$ such that for every model $M$ of $T$ and every $\ell$-split order
  $(f,{\preceq})\in\cS_{\ell,V(M)}$ on $V(M)$ with $\thk(f)\geq n$, there exists an $\ell$-Ramsey pattern
  $Q\in\cP_{\ell,\cL}$ over $\cL$ and a set $W\subseteq V(M)$ such that $\thk(f\rest_W)\geq m$ and $M\rest_W$
  is $Q$-uniform with respect to $(f\rest_W, {\preceq_W})$.
\end{definition}

Note that since $\cL$ is finite, there are only finitely many $\ell$-Ramsey patterns on $\cL$. Note also that
the definition of $R_{\ell,T}(m)$ is strong in the sense that \emph{every} $\ell$-split order of $V(M)$ is
required to yield a uniform submodel. This is slightly stronger than our motivating example of tournaments:
our definition for $\TTournament$ with $\ell=1$ requires that every ordering $\leq$ of the vertices of $M$
yields a tournament of size $m$ whose edges either all match the order $\leq$ or all disagree with $\leq$.

\begin{example}
  In the language $\cL$ containing a single predicate symbol $E$ of arity $k(E) = 2$, for every $n\geq 2$,
  there are exactly three (up to isomorphism) canonical structures $M$ of size $n$ that are $Q$-uniform for
  some $1$-Ramsey pattern $Q\in\cP_{1,\cL}$ with respect to some $(f,{\preceq})\in\cS_{1,V(M)}$: the complete
  graph $K_n$, the empty graph $\overline{K}_n$ and the transitive tournament $\Tr_n$. Note also that for
  $n\geq 2$, both $\cU_1(K_n)$ and $\cU_1(\overline{K}_n)$ have a single element but $\cU_1(\Tr_n)$ has two
  elements.

  In the same language, canonical structures $M$ that are $Q$-uniform for some $\ell$-Ramsey pattern $Q$ with
  respect to some $(f,{\preceq})$ are precisely those in which each level set $f^{-1}(i)$ of $f$ induces
  either a complete graph $K_{\lvert f^{-1}(i)\rvert}$, an empty graph $\overline{K}_{\lvert f^{-1}(i)\rvert}$
  or a transitive tournament $\Tr_{\lvert f^{-1}(i)\rvert}$ and (directed) edges between $v,w\in V(M)$ in
  different level sets of $f$ are completely determined by $f(v)$ and $f(w)$. See Figure~\ref{fig:uniform}.
\end{example}

\begin{figure}[htb]
  \begin{center}
    \begingroup
\def\sep{4}
\def\smallradius{0.9cm}
\def\radius{1cm}
\def\shift{0.3}
\def\lwidth{0.2cm}

\begin{tikzpicture}
  \coordinate (A) at (0,0);
  \coordinate (B) at (\sep,0);
  \coordinate (C) at (\sep,-\sep);
  \coordinate (D) at (0,-\sep);

  \draw[shorten <=\smallradius, shorten >=\radius, arrows={-latex}, line width=\lwidth] (A) -- (B);
  \draw[shorten <=\smallradius, shorten >=\radius, arrows={-latex}, line width=\lwidth] (B) -- (C);
  \draw[shorten <=\smallradius, shorten >=\radius, arrows={-latex}, line width=\lwidth]
  ($(C) + (0,\shift)$) -- ($(D) + (0,\shift)$);
  \draw[shorten <=\smallradius, shorten >=\radius, arrows={-latex}, line width=\lwidth]
  ($(D) + (0,-\shift)$) -- ($(C) + (0,-\shift)$);
  \draw[shorten <=\smallradius, shorten >=\radius, arrows={-latex}, line width=\lwidth] (A) -- (C);
  \draw[shorten <=\smallradius, shorten >=\radius, arrows={-latex}, line width=\lwidth]
  ($(A) + (\shift,0)$) -- ($(D) + (\shift,0)$);
  \draw[shorten <=\smallradius, shorten >=\radius, arrows={-latex}, line width=\lwidth]
  ($(D) + (-\shift,0)$) -- ($(A) + (-\shift,0)$);

  \draw[fill=white] (A) circle (\radius);
  \draw[fill=white] (B) circle (\radius);
  \draw[fill=white] (C) circle (\radius);
  \draw[fill=white] (D) circle (\radius);

  \node at (A) {$K_{\lvert f^{-1}(1)\rvert}$};
  \node at (B) {$\overline{K}_{\lvert f^{-1}(2)\rvert}$};
  \node at (C) {$\Tr_{\lvert f^{-1}(3)\rvert}$};
  \node at (D) {$\Tr_{\lvert f^{-1}(4)\rvert}$};
\end{tikzpicture}

\endgroup
    \captionsetup{singlelinecheck=off}
    \caption[.]{Pictorial view of a $Q$-uniform model for the Ramsey pattern $Q\in\cP_{4,\{E\}}$ ($k(E) = 2$)
      given by
      \begin{align*}
        Q_E \df \{
        & ((1,1),{\leq}),((1,1),{\geq}),((3,3),{\leq}),((4,4),{\geq}),\\
        & ((1,2),{\preceq_0}), ((1,3),{\preceq_0}),((1,4),{\preceq_0}),\\
        & ((2,3),{\preceq_0}), ((3,4),{\preceq_0}),((4,3),{\preceq_0}),((4,1),{\preceq_0})\},
      \end{align*}
      where $\leq$ is the usual order on $[2]$, $\geq$ is its reverse and $\preceq_0$ is the trivial partial
      order on $[2]$, and the functions $f\:[2]\to[4]$ are represented as $(f(1),f(2))$. An arrow from a part
      $A$ to a part $B$ in the figure means that $(a,b)\in R_E(M)$ for every $a\in A$ and every $b\in B$.}
    \label{fig:uniform}
  \end{center}
\end{figure}

\begin{theorem}\label{thm:Ramsey}
  For every $\ell\in\NN_+$, every $m\in\NN$ and every canonical theory $T$, the $(\ell,T,m)$-Ramsey number
  $R_{\ell,T}(m)$ is finite.
\end{theorem}

We provide a proof of Theorem~\ref{thm:Ramsey} via a reduction to Theorem~\ref{thm:hypergraphRamsey} in
Section~\ref{sec:Ramsey}. Let us also note that the case $\ell=1$ of Theorem~\ref{thm:Ramsey} follows from the
very general Ramsey Theory for systems of~\cite{NR89}.

We will typically be working in theories of the form $\TGraph\cup T$ and two types of Ramsey patterns will
play an important role in the alternative formula for the abstract chromatic number.

\begin{definition}[Complete patterns and Tur\'{a}n patterns]
  Fix $\ell\in\NN_+$ and a language $\cL$ and let $E\in\cL$ be a binary predicate symbol.

  A $1$-Ramsey pattern $Q\in\cP_{1,\cL}$ on $\cL$ is called \emph{$E$-complete} if $Q_E = \cS_{1,2}$. We let
  $\cC_\cL^E$ be the set of all $E$-complete $1$-Ramsey patterns on $\cL$.

  An $\ell$-Ramsey pattern $Q\in\cP_{\ell,\cL}$ on $\cL$ is called \emph{$E$-Tur\'{a}n} if
  \begin{align*}
    Q_E & = \{(g,{\preceq})\in\cS_{\ell,2} \mid g\text{ is injective}\}.
  \end{align*}
  We let $\cT_{\ell,\cL}^E$ be the set of all $E$-Tur\'{a}n $\ell$-Ramsey patterns on $\cL$.
\end{definition}

Note that if $I\:T_{\{E\}}\leadsto T_\cL$ is the structure-erasing interpretation, then $M$ is $Q$-uniform
with respect to some $(f,{\preceq})\in\cS_{1,V(M)}$ for some $E$-complete $Q\in\cC_\cL^E$ if and only if
$I(M)\cong K_{\lvert M\rvert}$. Analogously, $M$ is $Q$-uniform with respect to some
$(f,{\preceq})\in\cS_{\ell,V(M)}$ for some $E$-Tur\'{a}n $Q\in\cT_{\ell,\cL}^E$ if and only if $I(M)$ is a
complete $\ell$-partite graph with respect to the partition given by the level sets of $f$.

\subsection{Non-induced setting}

As we mentioned in the introduction, the abstract chromatic number works in the general setting of induced
submodels. For the non-induced setting, we will be able to provide a slightly simpler formula for the abstract
chromatic number in terms of proper split orderings defined below.

\begin{definition}[$E$-upward closures and proper split orderings]
  Let $\cL$ be a language and let $E$ be the predicate symbol corresponding to $\TGraph$ in the language
  $\cL\cup\{E\}$ of $\TGraph\cup T_\cL$.

  Given a family $\cF$ of models of $\TGraph\cup T_\cL$, the \emph{$E$-upward closure} of $\cF$ is the family
  $\cF\up^E$ of all $F'$ that can be obtained from some $F\in\cF$ by possibly adding edges, that is, all
  models $F'$ of $\TGraph\cup T_\cL$ such that there exists $F\in\cF$ with
  \begin{align*}
    V(F') & = V(F); &
    R_E(F') & \supseteq R_E(F); &
    R_P(F') & = R_P(F) \quad (P\in\cL).
  \end{align*}

  Let $I\:\TGraph\leadsto\TGraph\cup T_\cL$ and $J\:T\leadsto\TGraph\cup T_\cL$ be the structure-erasing
  interpretations. Given $\ell\in\NN_+$, an $\ell$-Ramsey pattern $Q\in\cP_{\ell,\cL}$ on $\cL$ and a model
  $M$ of $\TGraph\cup T_\cL$, an \emph{$E$-proper $Q$-split ordering} of $M$ is a split order
  $(f,{\preceq})\in\cS_{\ell,V(M)}$ such that $J(M)$ is $Q$-uniform with respect to $(f,{\preceq})$ and $f$ is
  a proper coloring of the graph $I(M)$. The \emph{$E$-proper $\ell$-split ordering set} of $M$ is the set
  $\chi_\ell^E(M)$ of all $\ell$-Ramsey patterns $Q\in\cP_{\ell,\cL}$ such that $M$ has an $E$-proper
  $Q$-split ordering. We extend this definition to a family $\cF$ of canonical structures as
  $\chi_\ell^E(\cF)\df\bigcup_{M\in\cF}\chi_\ell^E(M)$.
\end{definition}

Note that in the definition of proper $Q$-split orderings, the predicate symbol $E$ is excluded from the
uniformity condition. Note also that if the language $\cL$ is empty, then $\cP_{\ell,\cL}$ has a
unique element $Q$ and a proper $(Q,\ell)$-split ordering of $M$ consists of any $\ell$-split order
$(f,{\preceq})$ in which $f$ is a proper coloring of the graph $I(M)$.

\section{Main results}
\label{sec:main}

In this section we formalize the main results. We start with the generalization of Theorem~\ref{thm:ESSAS} to
the setting of open interpretations. The case when $t=2$ and $T$ is non-degenerate was done
in~\cite[Example~31]{CR19}.

\begin{theorem}\label{thm:abstractchromatic}
  Let $t\in\NN_+$ and let $I\:\TGraph\leadsto T$ be an open interpretation. Then
  \begin{align}\label{eq:Turandensityabstractchromatic}
    \pi^t_I
    & =
    \begin{dcases*}
      \prod_{i=1}^{t-1} \left(1 - \frac{j}{\chi(I)-1}\right), & if $\chi(I)\geq 2$;\\
      -\infty, & if $\chi(I)\leq 1$.
    \end{dcases*}
  \end{align}
\end{theorem}

The next theorem gives an alternative formula for the abstract chromatic number based on the Ramsey uniformity
sets of the forbidden models.

\begin{theorem}\label{thm:Ramseyabstractchromatic}
  Let $I\:\TGraph\leadsto T$ be an open interpretation and let $T'$ be the theory obtained from $\TGraph\cup
  T$ by adding the axiom
  \begin{align*}
    \forall x\forall y, E(x,y) & \tot I(E)(x,y).
  \end{align*}
  Let $\cL$ be the language of $T'$ and let $\cF$ be such that $T' = \Forb_{T_\cL}(\cF)$. Then
  \begin{align}\label{eq:Ramseyabstractchromaticmin}
    \chi(I)
    & =
    \begin{dcases*}    
      \infty, & if $\cC_\cL^E\not\subseteq\cU_1(\cF)$;\\
      \min\{\ell\in\NN_+ \mid\cT_{\ell,\cL}^E\subseteq\cU_\ell(\cF)\}, & otherwise.
    \end{dcases*}
  \end{align}

  Furthermore, if $T$ is itself obtained from $\TGraph\cup T$ by adding axioms and $I$ acts identically on
  $E$, then the same result holds by taking $T' = T$ instead.
\end{theorem}

\begin{remark}\label{rmk:Ramseyabstractchromatic}
  In fact, we show that the set in~\eqref{eq:Ramseyabstractchromaticmin} is either empty or an infinite
  interval of $\NN_+$ (with the empty case only happening when $\chi(I)=\infty$), and thus we also have
  \begin{align}\label{eq:Ramseyabstractchromaticmax}
    \chi(I)
    & =
    \begin{dcases*}
      \infty, & if $\cC_\cL^E\not\subseteq\cU_1(\cF)$;\\
      \max\{\ell\in\NN_+ \mid\cT_{\ell,\cL}^E\not\subseteq\cU_\ell(\cF)\}\cup\{0\} + 1, & otherwise.
    \end{dcases*}
  \end{align}
\end{remark}

The alternative formula provided by the theorem above can be used to algorithmically compute $\chi(I)$ when
$T$ is finitely axiomatizable.

\begin{theorem}\label{thm:computability}
  There exists an algorithm that computes $(\chi(I),\pi^t_I)$ for $I\:\TGraph\leadsto T$ for a finitely
  axiomatizable $T$ from a list of the axioms of $T$, a description of $I$ and $t\in\NN$.
\end{theorem}

For the case when the theory is the theory of graphs with extra structure with some forbidden submodels that
are \emph{non-induced} in the graph part, we can provide slightly simpler formulas for $\chi(I)$. The first
theorem provides a formula based on the usual chromatic number, but as abstract
as~\eqref{eq:abstractchromatic} and the second provides formulas in terms of proper split orderings.

\begin{theorem}\label{thm:noninduced}
  Let $\cL$ be a language, let $E$ be the predicate symbol corresponding to $\TGraph$ in the language
  $\cL\cup\{E\}$ of $\TGraph\cup T_\cL$. Let $\cF$ be a family of models of $\TGraph\cup T_\cL$ and let
  $I\:\TGraph\leadsto\Forb_{\TGraph\cup T_\cL}(\cF\up^E)$ act identically on $E$.

  Then we have
  \begin{align}\label{eq:abstractchromaticnoninduced}
    \chi(I)
    & =
    \inf\{\chi(G) \mid
    G\in\cM[\TGraph]\land\forall M\in\cM[\Forb_{\TGraph\cup T_\cL}(\cF\up^E)], I(M)\not\cong G\}.
  \end{align}
\end{theorem}

\begin{theorem}\label{thm:Ramseynoninduced}
  Let $\cL$ be a language, let $E$ be the predicate symbol corresponding to $\TGraph$ in the language
  $\cL\cup\{E\}$ of $\TGraph\cup T_\cL$ and let $J\:T_\cL\leadsto\TGraph\cup T_\cL$ be the structure-erasing
  interpretation. Let $\cF$ be a family of models of $\TGraph\cup T_\cL$ and let
  $I\:\TGraph\leadsto\Forb_{\TGraph\cup T_\cL}(\cF\up^E)$ act identically on $E$.

  Then we have
  \begin{align}\label{eq:Ramseyabstractchromaticnoninducedinf}
    \chi(I)
    & =
    \inf\{\ell\in\NN_+ \mid \cP_{\ell,\cL}\subseteq\chi_\ell^E(\cF)\}.
  \end{align}

  Furthermore, we have $\chi(I) < \infty$ if and only if $\cP_{1,\cL}\subseteq\cU_1(J(\cF))$,
  where $J(\cF)\df\{J(F) \mid F\in\cF\}$.
\end{theorem}

\begin{remark}\label{rmk:Ramseynoninduced}
  Just as in the case of Theorem~\ref{thm:Ramseyabstractchromatic}, the set
  in~\eqref{eq:Ramseyabstractchromaticnoninducedinf} is either empty or an infinite interval of $\NN_+$, and
  thus we also have
  \begin{align}\label{eq:Ramseyabstractchromaticnoninducedsup}
    \chi(I)
    & =
    \sup\{\ell\in\NN_+ \mid
    \cP_{\ell,\cL}\not\subseteq\chi_\ell^E(\cF)\}\cup\{0\} + 1
  \end{align}
\end{remark}

\section{Abstract Tur\'{a}n densities from abstract chromatic number}
\label{sec:Turan}

The objective of this section is to prove Theorem~\ref{thm:abstractchromatic}. Before we do so, we show that the
set in the definition of $\chi(I)$ in~\eqref{eq:abstractchromatic} is a non-empty initial interval of $\NN$.

\begin{lemma}\label{lem:abstractchromset}
  Given an open interpretation $I\:\TGraph\leadsto T$, the set
  \begin{align}\label{eq:abstractchromset}
    \{\ell\in\NN_+ : \forall n\in\NN, \exists N\in\cM_n[T], I(N)\supseteq T_{n,\ell}\}\cup\{0\}
  \end{align}
  is a non-empty initial interval of $\NN$.

  In particular, we have
  \begin{align}\label{eq:abstractchromatic2}
    \chi(I)
    & =
    \inf\{\ell\in\NN_+ \mid \exists n\in\NN, \forall N\in\cM_n[T], I(N)\not\supseteq T_{n,\ell}\}.
  \end{align}
\end{lemma}

\begin{proof}
  Let $X$ be the set in~\eqref{eq:abstractchromset}. It is clear that $0\in X$. On the other hand, if $\ell\in
  X\cap\NN_+$, then for every $n\in\NN$, there exists $N\in\cM_n[T]$ such that $I(N)\supseteq T_{n,\ell}$. So
  if $\ell'\in [\ell]$ and $n\in\NN$, then since $T_{n,\ell'} \subseteq T_{\ell\cdot\ceil{n/\ell'}, \ell}$,
  it follows that there exists $N'\in\cM_n[T]$ such that $I(N')\supseteq T_{n,\ell'}$, hence $\ell'\in X$.

  Since $\chi(I) = \sup X + 1$ by~\eqref{eq:abstractchromatic} and $X$ is a non-empty initial interval of
  $\NN$, we get $\chi(I) = \inf \NN\setminus X$, so~\eqref{eq:abstractchromatic2} follows.
\end{proof}

\begin{proofof}{Theorem~\ref{thm:abstractchromatic}}
  If $\chi(I) = \infty$, then for every $n\in\NN$, there exists $N_n\in\cM_n[T]$ such that $I(N_n)\supseteq
  T_{n,n} = K_n$, so $\pi^t_I = 1$, hence~\eqref{eq:Turandensityabstractchromatic} holds.

  On the other hand, if $\chi(I) = 1$, then by Lemma~\ref{lem:abstractchromset}, there exists $n\in\NN$ such
  that for every $N\in\cM_n[T]$, we have $I(N)\not\supseteq T_{n,1} = \overline{K}_n$. But since every graph
  on $n$ vertices contains a non-induced copy of $\overline{K}_n$, we must have $\cM_n[T] = \varnothing$. This
  means that $T$ is degenerate, hence $\pi^t_I = -\infty$, so~\eqref{eq:Turandensityabstractchromatic} holds.

  Suppose then that $2\leq\chi(I) < \infty$. For every $n\in\NN$, let $N_n\in\cM_n[T]$ be such that
  $I(N_n)\supseteq T_{n,\chi(I)-1}$. Then we get
  \begin{align*}
    \pi^t_I
    & \geq
    \liminf_{n\to\infty} p(K_t,I(N_n))
    \geq
    \liminf_{n\to\infty} p(K_t,T_{n,\chi(I)-1})
    =
    \prod_{j=0}^{t-1}\left(1 - \frac{j}{\chi(I)-1}\right).
  \end{align*}
  Suppose now toward a contradiction that $(N_m)_{m\in\NN}$ is a sequence of models of $T$ with $\lvert
  N_m\rvert < \lvert N_{m+1}\rvert$ such that $\lim_{m\to\infty} p(K_t, I(N_m)) > \prod_{j=0}^{t-1}(1 -
  j/(\chi(I)-1))$. Fix $n\in\NN$ and note that Theorem~\ref{thm:ESSAS} for $\cF\df\{T_{n,\chi(I)}\}$ implies
  that there exists $m_n\in\NN$ such that $I(N_{m_n})\supseteq T_{n,\chi(I)}$. By restricting $N_{m_n}$ to a
  set $V$ of size $n$ such that $I(N_{m_n})\rest_V\supseteq T_{n,\chi(I)}$, we conclude that there exists
  $N'_n\in\cM_n[T]$ such that $I(N'_n)\supseteq T_{n,\chi(I)}$ so $\chi(I)\geq \chi(I)+1$, a contradiction (as
  $\chi(I) < \infty$).
\end{proofof}

\section{Partite Ramsey numbers}
\label{sec:Ramsey}

The objective of this section is to prove Theorems~\ref{thm:hypergraphRamsey} and~\ref{thm:Ramsey}.

\begin{proofof}{Theorem~\ref{thm:hypergraphRamsey}}
  The proof is by induction in the length $t$ of the tuple $\vec{k} = (k_1,\ldots,k_t)$.

  For the case $t = 1$, let us denote $k_1$ simply by $k$ and let us identify $\cP_{\ell,\vec{k}}$ with
  $2^{\cC_{\ell,k}}$. Let $c\df\lvert\cP_{\ell,\vec{k}}\rvert < \infty$ and let $n\df R(k,c,\ell m) < \infty$
  be the usual Ramsey number corresponding to finding monochromatic cliques of size $\ell m$ in colorings of
  $k$-uniform complete hypergraphs with $c$ colors. We will show that $R_{\ell,\vec{k}}(m)\leq n$.

  Suppose $H$ is a $\vec{k}$-hypergraph and $f\:V(H)\to[\ell]$ has $\thk(f)\geq n$. For every $j\in[\ell]$,
  let $v(1,j),\ldots, v(n,j)$ be distinct vertices in $f^{-1}(j)$ and let $V\df\{v(i,j) \mid i\in[n]\land
  j\in[\ell]\}$.

  For a set $A\in\binom{[n]}{k}$, let $\iota_A\:[k]\rightarrowtail [n]$ be the injective function that
  enumerates $A$ in increasing order and if we are further given a weak composition $q =
  (q_j)_{j=1}^\ell\in\cC_{\ell,k}$, let $A_q\subseteq V$ be defined by
  \begin{align}\label{eq:Aq}
    A_q
    & \df
    \left\{v(\iota_A(i),j) \;\middle\vert\;
    i\in[k]\land j\in[\ell]\land \sum_{r=1}^{j-1} q_r < i\leq \sum_{r=1}^j q_r\right\}.
  \end{align}
  Note that $\lvert A_q\rvert = k$ and $\lvert f^{-1}(j)\cap A_q\rvert = q_j$ for every
  $j\in[\ell]$. Furthermore, if $q\neq q'$, then $A_q\neq A_{q'}$.

  Define the coloring $g\:\binom{[n]}{k}\to\cP_{\ell,\vec{k}}$ by letting
  \begin{align*}
    g(A) \df \{q\in\cC_{\ell,k} \mid A_q \in E(H)\},
  \end{align*}
  where $E(H)$ is the edge set of $H$. By the definition of $n=R(k,c,\ell m)$, there exists $U\subseteq[n]$ of
  size $\lvert U\rvert = \ell m$ such that $g\rest_{\binom{U}{k}}$ is monochromatic, say, of color
  $Q\in\cP_{\ell,\vec{k}}$.

  Let us enumerate the elements of $U$ in increasing order $u_1 < \cdots < u_{\ell m}$ and let
  \begin{align*}
    W & \df \{v(u_{(j-1)m + r}, j) \mid j\in[\ell]\land r\in[m]\}.
  \end{align*}
  Clearly, for every $j\in[\ell]$, we have $W\cap f^{-1}(j) = \{v(u_{(j-1)m + r}, j) \mid r\in[m]\}$, which has
  size $m$, so $\thk(f\rest_W) = m$.

  We claim that $H\rest_W$ is $Q$-uniform with respect to $f\rest_W$. To show this, we need to show that for
  every $B\in\binom{W}{k}$, we have
  \begin{align}\label{eq:Bobjective}
    B\in E(H) & \iff q^B\in Q,
  \end{align}
  where $q^B\in\cC_{\ell,k}$ is given by $q^B_j\df\lvert f^{-1}(j)\cap B\rvert$. 

  Note that the definition of $W$ implies that there exists an increasing function
  $\eta_B\:[k]\rightarrowtail[n]$ with $\im(\eta_B)\subseteq U$ such that
  \begin{align*}
    B
    & =
    \left\{v(\eta_B(t), j) \;\middle\vert\;
    i\in[k]\land j\in[\ell]\land \sum_{r=1}^{j-1} q^B_r < i\leq \sum_{r=1}^j q^B_r\right\}.
  \end{align*}
  Since $\iota_{\im(\eta_B)} = \eta_B$, from~\eqref{eq:Aq} we get $\im(\eta_B)_{q^B} = B$ and for every
  $q\in\cC_{\ell,k}\setminus\{q^B\}$, we have $\im(\eta_B)_q\neq B$. Since $g\rest_{\binom{U}{k}}$ is
  monochromatic of color $Q$, we have
  \begin{align*}
    Q & = g(\im(\eta_B)) = \{q\in\cC_{\ell,k} \mid \im(\eta_B)_q\in E(H)\},
  \end{align*}
  so~\eqref{eq:Bobjective} follows, concluding the proof of case $t=1$.

  \medskip

  Suppose now that $t\geq 2$ and, by inductive hypothesis, suppose $m'\df R_{\ell,(k_1,\ldots,k_{t-1})}(m)$ is
  finite. Let also $n\df R_{\ell,(k_t)}(m')$, which by the case above is also finite. We will show that
  $R_{\ell,\vec{k}}(m)\leq n$.

  Suppose $H$ is a $\vec{k}$-hypergraph and $f\:V(H)\to[\ell]$ has $\thk(f)\geq n$. By the definition of
  $n=R_{\ell,(k_t)}(m')$, there exists $Q'\in\cP_{\ell,(k_t)}$ and $W'\subseteq V(H)$ such that
  $\thk(f\rest_{W'})\geq m'$ and the $k_t$-hypergraph part of $H\rest_{W'}$ is $Q'$-uniform with respect to
  $f\rest_{W'}$. In turn, by the definition of $m'=R_{\ell,(k_1,\ldots,k_{t-1})}(m)$, there exists
  $Q''\in\cP_{\ell,(k_1,\ldots,k_{t-1})}$ and $W\subseteq W'$ such that $\thk(f\rest_W)\geq m$ and the
  $(k_1,\ldots,k_{t-1})$-hypergraph part of $H\rest_W$ is $Q''$-uniform with respect to $f\rest_W$. By
  letting $Q\in\cP_{\ell,\vec{k}}$ be given by
  \begin{align*}
    Q_j & \df
    \begin{dcases*}
      Q''_j, & if $j\in[t-1]$;\\
      Q', & if $j = t$;
    \end{dcases*}
  \end{align*}
  it follows that $H\rest_W$ is $Q$-uniform with respect to $f\rest_W$.
\end{proofof}

Before we can finally prove Theorem~\ref{thm:Ramsey}, we need one more definition.

\begin{definition}
  If $\leq$ is a total order on a set $V$ and $f\:V\to[\ell]$, we let ${\leq}\down_f\df
  {\leq}\cap\bigcup_{i\in[\ell]} f^{-1}(i)\times f^{-1}(i)$ be the restriction of $\leq$ to the level sets of
  $f$, that is, it is the unique partial order such that $(f,{\leq}\down_f)$ is an $\ell$-split order and
  $\leq$ is an extension of it.
\end{definition}

\begin{proofof}{Theorem~\ref{thm:Ramsey}}
  Consider the set
  \begin{align*}
    K & \df \{(P,\leq) \mid P\in\cL\land {\leq} \text{ is a total order on }[k(P)]\},
  \end{align*}
  enumerate the elements of $K$ as $(P_1,{\leq^1}),\ldots,(P_t,{\leq^t})$ and define $\vec{k} = (k_1,\ldots,k_t)$
  by letting $k_i\df k(P_i)$.

  Let $n\df R_{\ell,\vec{k}}(m)$, which is finite by Theorem~\ref{thm:hypergraphRamsey}. We claim that
  $R_{\ell,T}(m)\leq n$. Suppose $M$ is a model of $T$ and $(f,{\preceq})\in\cS_{\ell,V(M)}$ is an $\ell$-split
  order on $V(M)$ with $\thk(f)\geq n$. Define the relation $\leq$ on $V(M)$ by
  \begin{align*}
    v\leq w & \iff f(v) < f(w) \lor v\preceq w.
  \end{align*}
  Since $(f,{\preceq})$ is a split order, it follows that $\leq$ is a total order extending $\preceq$. Note
  that $f$ becomes non-decreasing with respect to $\leq$ on $V(M)$ and the usual order on $[\ell]$, that is,
  we have
  \begin{align}\label{eq:fnondecreasing}
    v\leq w & \to f(v)\leq f(w)
  \end{align}
  for every $v,w\in V(M)$.

  Define now the $\vec{k}$-hypergraph $H$ with vertex set $V(H)\df V(M)$ by letting the $i$-th edge set be
  \begin{align*}
    E_i(H)
    & \df
    \left\{A\in\binom{V(H)}{k_i} \;\middle\vert\; \iota_A^i\in R_{P_i}(M)\right\},
  \end{align*}
  where $\iota_A^i:[k(P_i)]\rightarrowtail V(M)$ is the unique function with $\im(\iota_A^i) = A$ that is
  increasing with respect to the order $\leq^i$ on $[k(P_i)]$ and the order $\leq$ on $V(M)$ (the latter
  condition is equivalent to ${\leq_{\iota_A^i}} = {\leq^i}$). For every $P\in\cL$, let $I_P \df \{i\in[t] \mid P_i =
  P\}$ and note that
  \begin{align}\label{eq:RamseyRPM}
    R_P(M) & = \{\alpha\in (V(M))_{k(P)} \mid i\in I_P\land \im(\alpha)\in E_i(H)\land {\leq_\alpha} = {\leq^i}\}.
  \end{align}

  By the definition of $n=R_{\ell,\vec{k}}(m)$, there exists $Q'\in\cP_{\ell,\vec{k}}$ and a set $W\subseteq
  V(H)$ such that $\thk(f\rest_W)\geq m$ and $H\rest_W$ is $Q'$-uniform with respect to $f\rest_W$. Define
  then the $\ell$-Ramsey pattern $Q\in\cP_{\ell,\cL}$ on $\cL$ by
  \begin{align}\label{eq:RamseyQP}
    Q_P & \df \{(g,{\leq^i}\down_g) \mid g\:[k(P)]\to[\ell]\land q^g\in Q_i'\land i\in I_P^g\},
  \end{align}
  where $q^g\in\cC_{\ell,k(P)}$ is the weak composition given by $q^g_j \df\lvert g^{-1}(j)\rvert$ and
  \begin{align*}
    I^g_P & \df \{i\in I_P \mid \forall j_1,j_2\in[k(P)], (j_1\leq^i j_2 \to g(j_1)\leq g(j_2))\}.
  \end{align*}

  We claim that $M\rest_W$ is $Q$-uniform with respect to $(f\rest_W, {\preceq_W})$. To show this, we have to
  show that
  \begin{align*}
    R_P(M\rest_W) & = \{\alpha\in (W)_{k(P)} \mid (f\comp\alpha, {\preceq_\alpha})\in Q_P\}.
  \end{align*}

  Let $\alpha\in R_P(M\rest_W)$ and let us show that $(f\comp\alpha,{\preceq_\alpha})\in
  Q_P$. By~\eqref{eq:RamseyRPM}, there exists $i\in I_P$ such that $\im(\alpha)\in E_i(H)$ and ${\leq_\alpha}
  = {\leq^i}$. Note that if $j_1,j_2\in [k(P)]$ are such that $j_1\leq^i j_2$, then we must have
  $\alpha(j_1)\leq \alpha(j_2)$, hence~\eqref{eq:fnondecreasing} implies $f(\alpha(j_1))\leq f(\alpha(j_2))$,
  so $i\in I_P^{f\comp\alpha}$. On the other hand, since $\leq$ extends $\preceq$ and $(f,{\preceq})$ is a
  split order, it follows that ${\preceq_\alpha} = {\leq^i}\down_{f\comp\alpha}$. Note also that since
  $H\rest_W$ is $Q'$-uniform with respect to $f\rest_W$ and $\im(\alpha)\in E_i(H)$, we must have
  $q^{f\comp\alpha}\in Q'_i$. Putting everything together, we have that there exists $i\in I_P^{f\comp\alpha}$
  such that $q^{f\comp\alpha}\in Q'_i$ and ${\preceq_\alpha} = {\leq^i}\down_{f\comp\alpha}$,
  so~\eqref{eq:RamseyQP} gives $(f\comp\alpha,{\preceq_\alpha})\in Q_P$.

  Suppose now that $\alpha\in (W)_{k(P)}$ is such that $(f\comp\alpha, {\preceq_\alpha})\in Q_P$ and let us
  show that $\alpha\in R_P(M\rest_W)$. From~\eqref{eq:RamseyQP}, we know that there exists $i\in
  I_P^{f\comp\alpha}$ such that $q^{f\comp\alpha}\in Q'_i$ and ${\preceq_\alpha} =
  {\leq^i}\down_{f\comp\alpha}$. The fact that $H\rest_W$ is $Q'$-uniform with respect to $f\rest_W$ then
  implies that $\im(\alpha)\in E_i(H)$ and the fact that $i\in I_P^{f\comp\alpha}$ along
  with~\eqref{eq:fnondecreasing} implies ${\leq_\alpha} = {\leq^i}$. Putting everything together, since
  $I_P^{f\comp\alpha}\subseteq I_P$, we have that there exists $i\in I_P$ such that $\im(\alpha)\in E_i(H)$
  and ${\leq_\alpha} = {\leq^i}$, so by~\eqref{eq:RamseyRPM}, we get $\alpha\in R_P(M\rest_W)$.

  Therefore $M\rest_W$ is $Q$-uniform with respect to $(f\rest_W, {\preceq_W})$.
\end{proofof}

\section{Ramsey-based formula for the abstract chromatic number}
\label{sec:Ramseyabstractchromatic}

In this section we prove Theorems~\ref{thm:Ramseyabstractchromatic} and~\ref{thm:computability}.

\begin{proofof}{Theorem~\ref{thm:Ramseyabstractchromatic}}
  Recall from~\cite[Remark~2]{CR19} that we can write $I = J\comp A\comp S$, where
  $S\:\TGraph\leadsto\TGraph\cup T$ is the structure-erasing interpretation, $A\:\TGraph\cup T\leadsto T'$ is
  the axiom-adding interpretation and $J\:T'\leadsto T$ is the isomorphism that acts identically on predicate
  symbols of $T$ and acts as $I$ on $E$ (the inverse $J^{-1}\:T\leadsto T'$ acts identically on the predicate
  symbols of $T$).

  We start by characterizing when $\chi(I)$ is finite. Suppose first that $\cC_\cL^E\not\subseteq\cU_1(\cF)$
  and let us show that $\chi(I)=\infty$. Let $Q\in\cC_\cL^E\setminus\cU_1(\cF)$ and for every $n\in\NN$, let
  $N_n$ be the unique structure on $\cL$ with vertex set $[n]$ that is $Q$-uniform with respect to the usual
  order $\leq$ on $[n]$, that is, we have
  \begin{align*}
    R_P(N_n) & \df \{\alpha\in ([n])_{k(P)} \mid {\leq_\alpha}\in Q_P\}.
  \end{align*}
  Our choice of $Q$ ensures that $N_n$ is a model of $T' = \Forb_{T_\cL}(\cF)$. Since $Q\in\cC_\cL^E$, it
  follows that $S(A(N_n))$ is the complete graph $K_n$, so $I(J^{-1}(N_n))\supseteq T_{n,\ell}$ for
  every $\ell\in\NN_+$, so $\chi(I) = \infty$ by~\eqref{eq:abstractchromatic}.

  \medskip

  Suppose now that $\cC_\cL^E\subseteq\cU_1(\cF)$ and let us show that $\chi(I) < \infty$ and that the minimum
  in~\eqref{eq:Ramseyabstractchromaticmin} is attained (i.e., that the set
  in~\eqref{eq:Ramseyabstractchromaticmin} is non-empty). For every
  $Q\in\cC_\cL^E$, let $F_Q\in\cF$ and $\preceq^Q\in\cS_{1,V(F_Q)}$ be such that $F_Q$ is $Q$-uniform with
  respect to $\preceq^Q$. Let $m\df\max\{\lvert F_Q\rvert \mid Q\in\cC_\cL^E\}\cup\{2\}$ and let $n\df
  R_{1,T_\cL}(m)$ (which is finite by Theorem~\ref{thm:Ramsey}).

  We will show that $\chi(I)\leq n$. By~\eqref{eq:abstractchromatic2} of Lemma~\ref{lem:abstractchromset}, it
  is enough to show that every $N\in\cM_n[T]$ satisfies $I(N)\not\supseteq T_{n,n}$. Suppose not and for a
  violating $N$ let $M = J(N)$ be the associated model of $T'$. The definition of $n = R_{1,T_\cL}(m)$ implies
  that there exists $W\subseteq V(M)$ such that $\lvert W\rvert\geq m$ and $M\rest_W$ is $Q$-uniform with
  respect to $\leq_W$ where $\leq$ is the usual order over $[n]$. Since $I(N)\supseteq T_{n,n} = K_n$ and
  $m\geq 2$, it follows that $Q\in\cC_\cL^E$. But this is a contradiction as $M\rest_W$ must then contain an
  induced copy of $F_Q\in\cF$ (as $\lvert W\rvert\geq m\geq\lvert F_Q\rvert$), hence $\chi(I) < \infty$.

  To show that the minimum in~\eqref{eq:Ramseyabstractchromaticmin} is attained, it is enough to show that for
  $\ell\geq n$, we have $\cT_{\ell,\cL}^E\subseteq\cU_\ell(\cF)$ (as this implies that the set
  in~\eqref{eq:Ramseyabstractchromaticmin} is non-empty). Fix $Q\in\cT_{\ell,\cL}^E$ and let $N_Q$ be the
  unique structure on $\cL$ with vertex set $[\ell]$ that is $Q$-uniform with respect to the unique element of
  $\cS_{\ell,\ell}$ of the form $(\id_\ell,{\preceq_0})$, where $\id_\ell(i)\df i$ for every $i\in[\ell]$ and
  $\preceq_0$ is the trivial partial order, that is, we have
  \begin{align*}
    R_P(N_Q) & \df \{\alpha\in ([\ell])_{k(P)} \mid (\alpha, {\preceq_0})\in Q_P\}.
  \end{align*}
  Since $\ell\geq n = R_{1,T_\cL}(m)$, we know that there exists $W\subseteq[\ell]$ with $\lvert W\rvert = m$
  and some $Q'\in\cP_{1,\cL}$ such that $N_Q\rest_W$ is $Q'$-uniform with respect to $\leq_W$, where $\leq$ is
  the usual order on $[\ell]$. Since $N_Q$ is $Q$-uniform with respect to $(\id_\ell,{\preceq_0})$, $Q$ is an
  $E$-Tur\'{a}n pattern and $m\geq 2$, it follows that $Q'$ must be $E$-complete. But then since $\lvert
  F_{Q'}\rvert\leq m = \lvert W\rvert$, there exists $U\subseteq W$ such that $N_Q\rest_U\cong F_{Q'}$. As
  $F_{Q'}$ is an induced submodel of $N_Q$ and $Q\in\cU_\ell(N_Q)$, we get $Q\in\cU_\ell(F_{Q'})$, hence
  $\cT_{\ell,\cL}^E\subseteq\cU_\ell(\cF)$, so the minimum in~\eqref{eq:Ramseyabstractchromaticmin} is
  attained.

  \medskip

  To finish the proof of~\eqref{eq:Ramseyabstractchromaticmin}, it remains to show that if $\chi(I) < \infty$
  and $\ell_0 < \infty$ is the minimum in~\eqref{eq:Ramseyabstractchromaticmin}, then $\chi(I) = \ell_0$. We
  start by showing $\chi(I)\leq\ell_0$.

  Since $\cC_\cL^E$ is finite and $\cC_\cL^E\subseteq\cU_1(\cF)$, we know there exists a finite
  $\cF'\subseteq\cF$ such that $\cC_\cL^E\subseteq\cU_1(\cF')$. Since $\ell_0 < \infty$, we have
  $\cT_{\ell_0,\cL}^E\subseteq\cU_{\ell_0}(\cF)$, that is, for every $Q\in\cT_{\ell_0,\cL}^E$, there
  exists $F_Q\in\cF$ and $(f_Q,{\preceq^Q})\in\cS_{\ell_0,V(F_Q)}$ such that $F_Q$ is $Q$-uniform with
  respect to $(f_Q,{\preceq^Q})$.

  Let
  \begin{align*}
    m & \df\max\{\lvert F_Q\rvert \mid Q\in\cT_{\ell_0}^E\}\cup\{\lvert F\rvert \mid F\in\cF'\}\cup\{2\}
  \end{align*}
  and let
  $n\df \ell_0\cdot R_{\ell_0,T_\cL}(m)$ (which is finite by
  Theorem~\ref{thm:Ramsey}). By~\eqref{eq:abstractchromatic2} of Lemma~\ref{lem:abstractchromset}, to show
  $\chi(I)\leq\ell_0$, it is enough to show that every $N\in\cM_n[T]$ satisfies $I(N)\not\supseteq
  T_{n,\ell_0}$. Suppose not and for a violating $N\in\cM_n[T]$, let $f_N\:V(N)\to[\ell_0]$ be a function
  whose level sets are the parts of the natural partition of $T_{n,\ell_0}$ so that $\thk(f_N) =
  n/\ell_0 = R_{\ell_0,T_\cL}(m)$.

  Let $M\df J(N)$ be the associated model of $T'$ and let $\preceq^N$ be any partial order such that
  $(f_N,{\preceq^N})$ is an $\ell_0$-split order. Since $\thk(f_N) = R_{\ell_0,T_\cL}(m)$, there exists
  an $\ell_0$-Ramsey pattern $Q\in\cP_{\ell_0,\cL}$ and some $W\subseteq[n]$ such that
  $\thk(f_N\rest_W)\geq m$ and $M\rest_W$ is $Q$-uniform with respect to $(f_N\rest_W,{\preceq^N_W})$.

  We claim that $Q$ is an $E$-Tur\'{a}n pattern. Suppose not. Since the definition of $f_N$ ensures that $Q_E$
  contains all $(g,{\preceq})\in\cS_{\ell,2}$ with $g$ injective, for $Q$ to not be an $E$-Tur\'{a}n pattern
  there must exist $i\in[\ell]$ such that $Q_E$ contains at least one of
  $(g_i,{\leq}),(g_i,{\geq})\in\cS_{\ell,2}$, where $g_i(1) = g_i(2) = i$ and $\leq$ is the usual order on
  $[2]$ and $\geq$ is its reverse. From the symmetry of $E$ and the fact that $\thk(f_N\rest_W)\geq m\geq 2$,
  it follows that $Q_E$ must in fact contain both $(g_i,{\leq})$ and $(g_i,{\geq})$. Let $Q'\in\cC_\cL^E$ be
  given by
  \begin{align*}
    Q'_E & \df \cS_{1,2}; &
    Q'_P & \df \{{\preceq} \mid (f,{\preceq})\in Q_P\land \im(f) = \{i\}\};
  \end{align*}
  for every $P\in\cL\setminus\{E\}$. Let $U\df f_N^{-1}(i)\cap W$ and note that $M\rest_U$ is $Q'$-uniform
  with respect to $\preceq_U$. Since $\lvert U\rvert\geq \thk(f_N\rest_W)\geq m\geq\max\{\lvert F\rvert\mid
  F\in\cF'\}$ and since $Q'\in\cC_\cL^E$, there exists $F\in\cF'$ such that $M\rest_U$ contains a copy of $F$,
  so $M$ is not a model of $T' = \Forb_{T_\cL}(\cF)$, a contradiction. Thus $Q$ must be an $E$-Tur\'{a}n
  pattern.

  Since $Q\in\cT_{\ell,\cL}^E$, it follows that $M\rest_W$ must contain an induced copy of $F_Q\in\cF$,
  namely, such copy can be produced by taking exactly $\lvert f_Q^{-1}(i)\rvert$ vertices in $f_N^{-1}(i)\cap
  W$ for each $i\in[\ell_0]$ (this is possible since $\lvert f_Q^{-1}(i)\rvert\leq\lvert F_Q\rvert\leq m\leq
  \thk(f_N\rest_W)$). This contradicts the fact that $M$ is a model of $T'=\Forb_{T_\cL}(\cF)$, hence
  $\chi(I)\leq\ell_0$.

  \medskip

  Let us now show that $\chi(I)\geq\ell_0$. If $\ell_0=1$, then the inequality trivially holds, so suppose
  $\ell_0\geq 2$. From the definition of $\ell_0$, there exists
  $Q\in\cT_{\ell_0-1,\cL}^E\setminus\cU_{\ell_0-1}(\cF)$. For every $n\in\NN$, let $f_n\:[n]\to[\ell_0-1]$ be any function
  with $\thk(f_n) = \floor{n/(\ell_0-1)}$ and let $N_n$ be the unique structure on $\cL$ with vertex set $[n]$ that
  is $Q$-uniform with respect to $(f_n,{\leq}\down_{f_n})$, where $\leq$ is the usual order on $[n]$, that is,
  we have
  \begin{align*}
    R_P(N_n) & \df \{\alpha\in ([n])_{k(P)} \mid (f_n\comp\alpha, ({\leq}\down_{f_n})_\alpha)\in Q_P\}.
  \end{align*}
  Our choice of $Q$ ensures that $N_n$ is a model of $T' = \Forb_{T_\cL}(\cF)$.
  
  Since $\thk(f_n) = \floor{n/(\ell_0-1)}$ and $Q\in\cT_{\ell_0-1,\cL}^E$, it follows that $S(A(N_n))$ is isomorphic to
  the Tur\'{a}n graph $T_{n,\ell_0-1}$, which implies that $I(J^{-1}(N_n))\cong T_{n,\ell_0-1}$, so
  by~\eqref{eq:abstractchromatic}, we have $\chi(I)\geq\ell_0$.

  This concludes the proof of~\eqref{eq:Ramseyabstractchromaticmin}.

  \medskip

  Finally, let us consider the case when $T$ is itself obtained from $\TGraph\cup T$ by adding axioms and $I$
  acts identically on the predicate symbol $E$ of $\TGraph$. To apply the previous case of the theorem, note
  that to form $\TGraph\cup T$, we add a new predicate symbol $E'$ corresponding to the new copy of $\TGraph$
  and the theory $T'$ is defined from $\TGraph\cup T$ by adding the axiom
  \begin{align*}
    \forall x\forall y, E'(x,y) & \tot E(x,y).
  \end{align*}
  But then the isomorphism $J\:T'\leadsto T$ simply copies $E$ to $E'$, which means that we can replace $T'$
  with $T$ and use $E$ from $T$ in place of the newly added $E'$ from $T'$.
\end{proofof}

\begin{remark}
  One of the consequences of Theorem~\ref{thm:Ramseyabstractchromatic} is that to compute $\chi(I)$, models
  $F\in\cF$ such that the graph part $I(J^{-1}(F))$ contains an induced copy of $\overline{P}_3$ (the graph on
  $3$ vertices with exactly $1$ edge) are completely irrelevant as such models are never uniform for
  complete patterns nor for Tur\'{a}n patterns.
\end{remark}

\begin{proofof}{Remark~\ref{rmk:Ramseyabstractchromatic}}
  We want to show that the set
  \begin{align*}
    X & \df \{\ell\in\NN_+ \mid\cT_{\ell,\cL}^E\subseteq\cU_\ell(\cF)\}
  \end{align*}
  in~\eqref{eq:Ramseyabstractchromaticmin} is either empty or an infinite interval of $\NN_+$. To show this,
  it is enough to show that if $\ell\in\NN_+\setminus X$ and $\ell'\in[\ell]$, then $\ell'\notin X$. But if
  $\ell\in \NN_+\setminus X$ then there exists $Q\in\cT_{\ell,\cL}^E\setminus\cU_\ell(\cF)$. Let then
  $Q'\in\cT_{\ell',\cL}^E$ be given by
  \begin{align*}
    Q'_P & \df \{(f,{\preceq})\in Q_P \mid \im(f)\subseteq[\ell']\} & (P\in\cL),
  \end{align*}
  where we reinterpret functions $f\:[k(P)]\to[\ell]$ with $\im(f)\subseteq[\ell']$ as
  $f\:[k(P)]\to[\ell']$. We claim that $Q'\notin\cU_{\ell'}(\cF)$. Indeed, if $F\in\cF$ was $Q'$-uniform with
  respect to some $(f,{\preceq})\in\cS_{\ell',V(F)}$, then it would also be $Q$-uniform with respect to
  $(\widehat{f},{\preceq})$, where $\widehat{f}$ is obtained from $f$ by simply extending the codomain to
  $[\ell]$. Hence $\ell'\notin X$.

  Since $X$ is either empty or an infinite interval of $\NN_+$, it follows that $\inf X = \sup\NN\setminus X +
  1$. If we further assume that $\chi(I) < \infty$, then $X$ is non-empty so $\min X = \max\NN\setminus X + 1$,
  hence~\eqref{eq:Ramseyabstractchromaticmin} and~\eqref{eq:Ramseyabstractchromaticmax} are equal.
\end{proofof}

Before showing Theorem~\ref{thm:computability}, let us first address a small technicality on axiomatization of
universal theories.

\begin{lemma}\label{lem:axioms}
  If $T$ be a universal theory that is finitely axiomatizable, then it has a finite axiomatization in which
  all of its axioms are universal. Furthermore, such finite axiomatization with universal axioms can be
  algorithmically computed from any finite axiomatization of $T$.
\end{lemma}

\begin{proof}
  Let $A$ be a finite list of axioms of $T$. Since $T$ is universal, the set $S$ of all universal formulas that
  are theorems of $T$ is an axiomatization of $T$, hence $S\vdash\bigwedge_{\phi\in A} \phi$, which implies
  that there must exist a finite set $S'$ such that $S'\vdash\bigwedge_{\phi\in A} \phi$, so $S'$ is a finite
  axiomatization of $T$ by universal formulas.

  To algorithmically compute $S'$ as above, we can enumerate all universal formulas $\phi$ that are theorems
  of $T$ in parallel (by also enumerating possible proofs of $\phi$ from $A$ in parallel) and also check in
  parallel whether finite subsets $S'$ of the $S$ enumerated so far satisfy $S'\vdash\bigwedge_{\phi\in
    A}\phi$ (by also enumerating possible proofs in parallel). The reasoning above shows that such algorithm
  must eventually find a satisfying $S'$.
\end{proof}

\begin{proofof}{Theorem~\ref{thm:computability}}
  Using the notation of Theorem~\ref{thm:Ramseyabstractchromatic}, note that the fact that $T$ is finitely
  axiomatizable implies that $T'$ is also finitely axiomatizable and the list of axioms of $T'$ can trivially
  be computed from the list of axioms of $T$ and a description of $I$. By Lemma~\ref{lem:axioms}, we may
  compute an axiomatization $A$ of $T'$ in which every axiom is a universal formula.

  Let $k$ be the maximum number of variables appearing in an axiom in $A$ and let $\cF$ be the (finite) set of
  all canonical structures $M$ on $\cL$ with vertex set $[t]$ for some $t\leq k$ that are not models of
  $T'$. Our choice of $k$ ensures that $T' = \Forb_{T_\cL}(\cF)$. We then check if
  $\cC_\cL^E\subseteq\cU_1(\cF)$. If this is false, then Theorem~\ref{thm:Ramseyabstractchromatic} guarantees
  that $\chi(I) = \infty$. Otherwise, we know that $\chi(I) < \infty$ and is given
  by~\eqref{eq:Ramseyabstractchromaticmin}, which means that we can compute it by finding the smallest
  $\ell\in\NN_+$ such that $\cT_{\ell,\cL}^E\subseteq\cU_\ell(\cF)$; Theorem~\ref{thm:Ramseyabstractchromatic}
  ensures that such $\ell$ exists and is precisely $\chi(I)$.

  Finally, we can compute $\pi^t_I$ from $\chi(I)$ and $t$ using
  formula~\eqref{eq:Turandensityabstractchromatic} in Theorem~\ref{thm:abstractchromatic}. Note that this is a
  valid algorithm as all sets and searches above are finite.
\end{proofof}

\section{The non-induced case}
\label{sec:noninduced}

In this section, we prove Theorems~\ref{thm:noninduced} and~\ref{thm:Ramseynoninduced}, which provide simpler
formulas for the abstract chromatic number in the setting of graphs with extra structure with some forbidden
submodels that are non-induced in the graph part.

For this section, let us fix a language $\cL$, let $E$ be the predicate symbol of $\TGraph$ in the language
$\cL\cup\{E\}$ of $\TGraph\cup T_\cL$, let $J\:T_\cL\leadsto\TGraph\cup T_\cL$ be the structure-erasing
interpretation and let $\cF$ be a family of models of $\TGraph\cup T_\cL$.

\begin{proofof}{Theorem~\ref{thm:noninduced}}
  Let $\ell_0$ be the right-hand side of~\eqref{eq:abstractchromaticnoninduced}.

  Note that if $G\in\cM[\TGraph]$ is such that for every $M\in\cM[\Forb_{\TGraph\cup T_\cL}(\cF\up^E)]$, we
  have $I(M)\not\cong G$. Since for $n\df\lvert G\rvert\chi(G)$, we have $T_{n,\chi(G)}\supseteq G$, from the
  definition of $\cF\up^E$, it follows that for every $M\in\cM[\Forb_{\TGraph\cup T_\cL}]$, we have
  $I(M)\not\supseteq T_{n,\chi(G)}$, so by~\eqref{eq:abstractchromatic2} of Lemma~\ref{lem:abstractchromset},
  we have $\chi(I)\leq\ell_0$.

  On the other hand, if $\ell\in\NN_+$ is such that there exists $n\in\NN_+$ such that for all
  $N\in\cM_n[\Forb_{\TGraph\cup T_\cL}(\cF\up^E)]$, we have $T_{n,\ell}\not\subseteq I(N)$, then we must also
  have that $I(N)\not\cong T_{n,\ell}$ for every $N\in\cM[\Forb_{\TGraph\cup T_\cL}(\cF\up^E)]$, hence
  from~\eqref{eq:abstractchromatic2} of Lemma~\ref{lem:abstractchromset} we also get $\ell_0\leq\chi(I)$.
\end{proofof}

To prove Theorem~\ref{thm:Ramseynoninduced}, we first need to relate uniformity of over $\cL\cup\{E\}$ with
uniformity and $E$-proper split orders over $\cL$.

\begin{claim}\label{clm:noninducedcomplete}
  For $Q\in\cC_{\cL\cup\{E\}}^E$, we have $Q\in\cU_1(\cF\up^E)$ if and only if
  $Q\rest_\cL\in\cU_1(J(\cF))$, where $Q\rest_\cL\in\cP_{1,\cL}$ is the restriction of $Q$ to
  $\cL$ and $J(\cF)\df\{J(F) \mid F\in\cF\}$.
\end{claim}

\begin{proof}
  Suppose $Q\in\cU_1(\cF\up^E)$, that is, there exists some $F\in\cF\up^E$ and some
  ${\preceq}\in\cS_{1,V(F)}$ such that $F$ is $Q$-uniform with respect to $\preceq$. From the definition of
  $\cF\up^E$, there exists $F'\in\cF$ such that $V(F')=V(F)$, $R_E(F')\subseteq R_E(F)$ and $R_P(F')=R_P(F)$
  for every $P\in\cL$. Since $F$ is $Q$-uniform with respect to $\preceq$, it follows that $J(F) = J(F')$ is
  $Q\rest_\cL$-uniform with respect to $\preceq$, so $Q\rest_\cL\in\cU_1(J(F'))$.

  \medskip

  Suppose now that $Q\rest_\cL\in\cU_1(J(\cF))$, that is, there exists some $F\in\cF$ and some
  ${\preceq}\in\cS_{1,V(F)}$ such that $J(F)$ is $Q\rest_\cL$-uniform with respect to $\preceq$. Let $F'$ be
  defined by $V(F')\df V(F)$, $R_P(F')\df R_P(F)$ for every $P\in\cL$ and $R_E(F')\df (V(F'))_2$. Note that
  $F'\in\cF\up^E$ and $F'$ is $Q$-uniform with respect to $\preceq$, so $Q\in\cU_1(F)$.
\end{proof}

\begin{claim}\label{clm:noninducedTuran}
  For $Q\in\cT_{\ell,\cL\cup\{E\}}^E$, we have $Q\in\cU_\ell(\cF\up^E)$ if and only if
  $Q\rest_\cL\in\chi_\ell^E(\cF)$, where $Q\rest_\cL\in\cP_{\ell,\cL}$ is the restriction of
  $Q$ to $\cL$.
\end{claim}

\begin{proof}
  Let $I\:\TGraph\leadsto\TGraph\cup T_\cL$ be the structure-erasing interpretation.
  
  Suppose $Q\in\cU_\ell(\cF\up^E)$, that is, there exists some $F\in\cF\up^E$ and some
  $(f,{\preceq})\in\cS_{\ell,V(F)}$ such that $F$ is $Q$-uniform with respect to $(f,{\preceq})$. From the
  definition of $\cF\up^E$, there exists $F'\in\cF$ such that $V(F')=V(F)$, $R_E(F')\subseteq R_E(F)$ and
  $R_P(F')=R_P(F)$ for every $P\in\cL$. Since $F$ is $Q$-uniform with respect to $(f,{\preceq})$, it follows
  that $J(F) = J(F')$ is $Q\rest_\cL$-uniform with respect to $(f,{\preceq})$. Since $Q$ is an $E$-Tur\'{a}n
  pattern, we also get that $f$ is a proper coloring of $I(F)$, hence also of $I(F')$, so
  $Q\rest_\cL\in\chi_\ell^E(F')$.

  \medskip

  Suppose now that $Q\rest_\cL\in\chi_\ell^E(\cF)$, that is, there exists some $F\in\cF$ and some $E$-proper
  $Q\rest_\cL$-split ordering $(f,{\preceq})\in\cS_{\ell,V(F)}$ of $F$. Define $F'$ by letting $V(F')\df
  V(F)$, $R_P(F')\df R_P(F)$ for every $P\in\cL$ and
  \begin{align*}
    R_E(F') & \df \{\alpha\in (V(F'))_2 \mid f(\alpha(1))\neq f(\alpha(2))\}.
  \end{align*}
  Note that since $f$ is a proper coloring of $I(F)$, it follows that $R_E(F')\supseteq R_E(F)$, so
  $F'\in\cF\up^E$. Note also that $F'$ is $Q$-uniform with respect to $(f,{\preceq})$ as $J(F)$ is
  $Q\rest_\cL$-uniform with respect to $(f,{\preceq})$, so $Q\in\cU_\ell(F')$.
\end{proof}

\begin{proofof}{Theorem~\ref{thm:Ramseynoninduced}}
  Note that first that the restriction function $\cC_{\cL\cup\{E\}}^E\to\cP_{1,\cL}$ given by $Q\mapsto
  Q\rest_\cL$ is bijective, so Claim~\ref{clm:noninducedcomplete} implies that
  $\cC_{\cL\cup\{E\}}^E\subseteq\cU_1(\cF)$ is equivalent to
  $\cP_{1,\cL}\subseteq\cU_1(J(\cF))$, so the characterization of $\chi(I) < \infty$ of
  Theorem~\ref{thm:Ramseynoninduced} follows from the characterization of $\chi(I) < \infty$ of
  Theorem~\ref{thm:Ramseyabstractchromatic}.

  On the other hand, the restriction function $\cT_{\ell,\cL\cup\{E\}}^E\to\cP_{\ell,\cL}$ given by $Q\mapsto
  Q\rest_\cL$ is also a bijection. This along with Claim~\ref{clm:noninducedTuran} implies that
  $\cT_{\ell,\cL\cup\{E\}}^E\not\subseteq\cU_\ell(\cF\up^E)$ is equivalent to
  $\cP_{\ell,\cL}^E\not\subseteq\chi_\ell^E(\cF)$, so from~\eqref{eq:Ramseyabstractchromaticmin} of
  Theorem~\ref{thm:Ramseyabstractchromatic}, we get that if $\chi(I) < \infty$,
  then~\eqref{eq:Ramseyabstractchromaticnoninducedinf} holds.

  It remains to prove that~\eqref{eq:Ramseyabstractchromaticnoninducedinf} also holds when $\chi(I) = \infty$,
  that is, we need to show that if $\cP_{1,\cL}\not\subseteq\cU_1(J(\cF))$, then
  $\cP_{\ell,\cL}\not\subseteq\chi_\ell^E(\cF)$ for every $\ell\in\NN_+$.

  Let $Q\in\cP_{1,\cL}\setminus\cU_1(J(\cF))$ and fix $\ell\in\NN_+$. Given
  $(f,{\preceq})\in\cS_{\ell,V}$, let $\preceq^f\in\cS_{1,V}$ be the total order on $V$ given by
  \begin{align*}
    v\preceq^f w & \iff f(v) < f(w)\lor v\preceq w.
  \end{align*}
  Clearly ${\preceq^f}\down_f = {\preceq}$.

  Let $Q'\in\cP_{\ell,\cL}$ be given by
  \begin{align*}
    Q'_P & \df \{(g,{\preceq})\in\cS_{\ell,k} \mid {\preceq^g}\in Q_P\}.
  \end{align*}
  We claim that $Q'\notin\chi_\ell^E(\cF)$. Suppose not, that is, suppose there exists $F\in\cF$
  and an $E$-proper $Q'$-split ordering $(f,{\preceq})\in\cS_{\ell,V(F)}$ of $F$. Note that for every
  $P\in\cL$, we have
  \begin{align*}
    R_P(F)
    & =
    \{\alpha\in (V(F))_{k(P)} \mid (f\comp\alpha, {\preceq_\alpha})\in Q'_P\}
    \\
    & =
    \{\alpha\in (V(F))_{k(P)} \mid ({\preceq_\alpha})^{f\comp\alpha}\in Q_P\}
    \\
    & =
    \{\alpha\in (V(F))_{k(P)} \mid ({\preceq^f})_\alpha\in Q_P\},
  \end{align*}
  hence $J(F)$ is $Q$-uniform with respect to $\preceq^f$, contradicting the fact that
  $Q\notin\cU_1(J(F))$. Hence $Q'\notin\chi_\ell^E(\cF)$ as desired.
\end{proofof}

\begin{proofof}{Remark~\ref{rmk:Ramseynoninduced}}
  In the proof above, we determined that $\cT_{\ell,\cL\cup\{E\}}^E\not\subseteq\cU_\ell(\cF\up^E)$ is
  equivalent to $\cP_{\ell,\cL}^E\not\subseteq\chi_\ell^E(\cF)$, so from
  Remark~\ref{rmk:Ramseyabstractchromatic} it follows that the set $X$
  in~\eqref{eq:Ramseyabstractchromaticnoninducedinf} is either empty or an infinite interval of $\NN_+$ and
  thus $\inf X = \sup\NN_+\setminus X + 1$.
\end{proofof}

\section{Applications to concrete theories}
\label{sec:concrete}

In this section we illustrate how to use the general theory to obtain easier formulas for the abstract
chromatic number for some specific theories. We start with the easy example of recovering the original setting
of Theorem~\ref{thm:ESSAS}: graphs with forbidden non-induced subgraphs.

\begin{proposition}\label{prop:usual}
  Let $\cF$ be a family of graphs and $\Forb_{\TGraph}^+(\cF)$ be the theory of all graphs that do not have
  any non-induced copy of graphs in $\cF$. Then for the axiom-adding interpretation
  $I_\cF^+\:\TGraph\leadsto\Forb_{\TGraph}^+(\cF)$, we have
  \begin{align*}
    \chi(I_\cF^+) & = \max\{\chi(\cF), 1\},
  \end{align*}
  where $\chi(\cF)\df\inf\{\chi(F) \mid F\in\cF\}$ is the infimum of the chromatic numbers of elements of
  $\cF$.
\end{proposition}

\begin{proof}
  Let $\cL\df\varnothing$ be the empty language and note that in the notation of Theorem~\ref{thm:Ramseynoninduced}
  we have $\Forb_{\TGraph}^+(\cF) = \Forb_{\TGraph\cup T_\cL}(\cF\up^E)$, so we get
  \begin{align*}
    \chi(I_\cF^+)
    & =
    \sup\{\ell\in\NN_+ \mid
    \cP_{\ell,\cL}\not\subseteq\chi_\ell^E(\cF)\}\cup\{0\} + 1.
  \end{align*}
  But since $\cL$ is empty, each $\cP_{\ell,\cL}$ has a unique element (namely, the empty pattern) and this
  unique element is in $\chi_\ell^E(F)$ if and only if there exists a proper coloring of $F$ with $\ell$
  colors, hence
  \begin{align*}
    \chi(I_\cF^+)
    & =
    \sup\{\ell\in\NN_+ \mid \forall F\in\cF, \ell < \chi(F)\}\cup\{0\} + 1
    =
    \max\{\chi(\cF), 1\},
  \end{align*}
  as desired.
\end{proof}

We now show how the picture changes when the forbidden subgraphs are induced instead of non-induced.

\begin{proposition}
  Let $\cF$ be a family of graphs and let $I_\cF\:\TGraph\leadsto\Forb_{\TGraph}(\cF)$ be the axiom-adding
  interpretation. If $\cF$ contains a complete graph, then
  \begin{align*}
    \chi(I_\cF)
    & =
    \max\{\ell\in\NN_+ \mid \cF\text{ does not contain a complete $\ell$-partite graph}\}\cup\{0\} + 1
    \\
    & =
    \min\{\ell\in\NN_+ \mid \cF\text{ contains a complete $\ell$-partite graph}\};
  \end{align*}
  otherwise, we have $\chi(I_\cF) = \infty$.
\end{proposition}

\begin{proof}
  In the notation of Theorem~\ref{thm:Ramseyabstractchromatic}, we can view $\Forb_{\TGraph}(\cF)$ as obtained
  from the theory $\TGraph\cup T_0$ by adding axioms, where $T_0$ is the trivial theory over the empty
  language. Then taking $T' = T$ (so $\cL=\{E\}$), note that $\cC_\cL^E$ contains a single element $Q_0$ and
  we have $Q_0\in\cU_1(F)$ if and only if $F$ is complete, so Theorem~\ref{thm:Ramseyabstractchromatic} gives
  $\chi(I_\cF) < \infty$ if and only if $\cF$ has a complete graph.

  Suppose then that $\cF$ contains a complete graph (so $\chi(I_\cF) < \infty$) and note that for every
  $\ell\in\NN_+$, $\cT_{\ell,\cL}^E$ also contains a single element $Q_\ell$ and we have
  $Q_\ell\in\cU_\ell(F)$ if and only if $F$ is a complete $\ell$-partite graph, hence
  from~\eqref{eq:Ramseyabstractchromaticmax} and~\eqref{eq:Ramseyabstractchromaticmin}, we get
  \begin{align*}
    \chi(I_\cF)
    & =
    \max\{\ell\in\NN_+ \mid \cF\text{ does not contain a complete $\ell$-partite graph}\}\cup\{0\} + 1
    \\
    & =
    \min\{\ell\in\NN_+ \mid \cF\text{ contains a complete $\ell$-partite graph}\},
  \end{align*}
  as desired.
\end{proof}

For our next example, we will recover the interval chromatic number used for ordered graphs in~\cite{PT06}
from our result.

\begin{definition}[Interval chromatic number~\cite{PT06}]
  An \emph{ordered graph} is a model of the theory $\TGraph\cup\TLinOrder$. A \emph{proper interval coloring}
  of an ordered graph $G$ is a proper coloring of the graph part of $G$ such that each color class is an
  interval of the order part of $G$. Formally, a proper interval coloring of $G$ is a function
  $f\:V(G)\to[\ell]$ such that
  \begin{gather*}
    \forall v,w\in V(G), (v,w)\in R_E(G)\implies f(v)\neq f(w);
    \\
    \forall u,v,w\in V(G), (u,v)\in R_<(G)\land (v,w)\in R_<(G)\land f(u) = f(w) \implies f(u) = f(v).
  \end{gather*}

  The \emph{interval chromatic number} $\chi_<(G)$ of an ordered graph $G$ is the minimum $\ell$ such that
  there exists a proper interval coloring of $G$ of the form $f\:V(G)\to[\ell]$.
\end{definition}

\begin{proposition}\label{prop:ordered}
  Let $\cF$ be a family of ordered graphs and $\Forb_{\TGraph\cup\TLinOrder}^+(\cF)$ be the theory of all
  ordered graphs that do not have any non-induced copy of ordered graphs in $\cF$. Then for the axiom-adding
  interpretation $I_\cF^<\:\TGraph\leadsto\Forb_{\TGraph\cup\TLinOrder}^+(\cF)$, we have
  \begin{align*}
    \chi(I_\cF^<) & = \max\{\chi_<(\cF), 1\},
  \end{align*}
  where $\chi_<(\cF)\df\inf\{\chi_<(F) \mid F\in\cF\}$ is the infimum of the interval chromatic numbers of
  elements of $\cF$.
\end{proposition}

\begin{proof}
  Let $\cL\df\{<\}$ with $k({<})\df 2$ and let further $F_1,F_2,F_3$ be the structures on $\{E\}\cup\cL$ defined
  by
  \begin{align*}
    \begin{aligned}
      V(F_1) & \df [2];\\
      R_E(F_1) & \df \varnothing;\\
      R_<(F_1) & \df \varnothing;
    \end{aligned}
    & &
    \begin{aligned}
      V(F_2) & \df [2];\\
      R_E(F_2) & \df \varnothing;\\
      R_<(F_2) & \df ([2])_2;
    \end{aligned}
    & &
    \begin{aligned}
      V(F_3) & \df [3];\\
      R_E(F_3) & \df \varnothing;\\
      R_<(F_3) & \df \{(1,2),(2,3),(3,1)\};
    \end{aligned}
  \end{align*}
  Define also $\widehat{\cF}\df\cF\cup\{F_1,F_2,F_3\}$ and note that in the notation of
  Theorem~\ref{thm:Ramseynoninduced}, we have $\Forb_{\TGraph\cup\TLinOrder}^+(\cF) = \Forb_{\TGraph\cup
    T_\cL}(\widehat{\cF}\up^E)$, so we get
  \begin{align*}
    \chi(I_\cF^<)
    & =
    \sup\{\ell\in\NN_+ \mid
    \cP_{\ell,\cL}\not\subseteq\chi_\ell^E(\widehat{\cF})\}\cup\{0\} + 1.
  \end{align*}

  For $i,j\in[\ell]$, let $\cS_{\ell,i,j}\df\{(f,{\preceq})\in\cS_{\ell,2} \mid \im(f)=\{i,j\}\}$. Note that
  $\cS_{\ell,i,j} = \cS_{\ell,j,i}$ and, regardless of whether $i\neq j$, we have $\lvert\cS_{\ell,i,j}\rvert =
  2$.

  Fix an $\ell$-Ramsey pattern $Q\in\cP_{\ell,\cL}$ on $\cL$. Let us call a pair $(i,j)\in[\ell]^2$
  \emph{empty in $Q$} if $\cS_{\ell,i,j}\cap Q_< = \varnothing$ and let us call $(i,j)$ \emph{full in $Q$} if
  $\cS_{\ell,i,j}\subseteq Q_<$.

  Note that if $(i,j)$ is empty in $Q\in\cP_{\ell,\cL}$, then any $(f,{\preceq})\in\cS_{\ell,2}$ with
  $\im(f)=\{i,j\}$ is an $E$-proper $Q$-split ordering of $F_1$. Conversely, note that if $(f,{\preceq})$ is
  an $E$-proper $Q$-split ordering of $F_1$ then $(f(1),f(2))$ is empty in $Q$. Hence
  $Q\in\cP_{\ell,\cL}$ has an empty pair if and only if $Q\in\chi_\ell^E(F_1)$. With an analogous argument, we
  can show that $Q\in\cP_{\ell,\cL}$ has a full pair if and only if $Q\in\chi_\ell^E(F_2)$.

  Let then $\cP_{\ell,\cL}'$ be the set of all $\ell$-Ramsey patterns that do not have any empty pairs nor any
  full pairs. To each $Q\in\cP_{\ell,\cL}'$, let us associate a tournament $T_Q$ given by $V(T_Q)\df[\ell]$
  and
  \begin{align*}
    E(T_Q)
    & \df
    \{(v,w)\in [\ell]^2 \mid v\neq w\land \exists (f,{\preceq})\in Q_<, f(1) = v\land f(2) = w\}.
  \end{align*}
  Note that the fact that $Q$ does not have any empty or full pairs ensures that $T_Q$ is indeed a tournament.

  We claim that for $Q\in\cP_{\ell,\cL}$, the tournament $T_Q$ has a cycle if and only if $Q\in\chi_\ell^E(F_3)$.

  For the forward direction, since $T_Q$ has a cycle, it must have a $3$-cycle, say $(u,v,w)\in[\ell]^3$ with
  $(u,v),(v,w),(w,u)\in E(T_Q)$. Then any $(f,{\preceq})\in\cS_{\ell,3}$ with $f(1) = u$, $f(2) = v$ and $f(3)
  = w$ is an $E$-proper $Q$-split ordering of $F_3$. For the backward direction, if
  $(f,{\preceq})\in\cS_{\ell,3}$ is an $E$-proper $Q$-split ordering of $F_3$, then $(f(1),f(2),f(3))$ is a
  $3$-cycle in $T_Q$.

  Let then $\cP_{\ell,\cL}''\df\{Q\in\cP_{\ell,\cL}' \mid T_Q\text{ is transitive}\}$ and note that our claims
  above show that
  \begin{align*}
    \chi(I_\cF^<)
    & =
    \sup\{\ell\in\NN_+ \mid
    \cP_{\ell,\cL}''\not\subseteq\chi_\ell^E(\cF)\}\cup\{0\} + 1.
  \end{align*}

  We now claim that for $Q\in\cP_{\ell,\cL}''$ and $F\in\cF$, we have $Q\in\chi_\ell^E(F)$ if and only if
  $\ell\geq\chi_<(F)$.

  For the forward direction, we claim that if $(f,{\preceq})\in\cS_{\ell,V(F)}$ is an $E$-proper $Q$-split
  ordering of $F$, then $f\:V(F)\to[\ell]$ is a proper interval coloring of $F$. Since $f$ is a proper
  coloring of the graph part of $F$, we need to show that its color classes are intervals of the order part
  of $F$. Suppose not, that is, suppose there exist $u,v,w\in V(F)$ such that $(u,v),(v,w)\in R_<(F)$ and
  $f(u) = f(w)\neq f(v)$. But then $(u,v)\in R_<(F)$ implies $(f(u),f(v))\in E(T_Q)$ and $(v,w)\in R_<(F)$
  implies $(f(v),f(w))\in E(T_Q)$, contradicting the fact that $T_Q$ does not have anti-parallel edges.

  For the backward direction, suppose $f\:V(F)\to[\ell]$ is a proper interval coloring of $F$. Since $T_Q$ is
  transitive, by possibly permuting the colors of $f$, we may suppose that the color classes of $f$ are in the
  same order in $F$ as the colors are in $T_Q$, that is, we may suppose that
  \begin{align}\label{eq:fTQ}
    \forall v,w\in V(F), (f(v)\neq f(w) & \to ((v,w)\in R_<(F)\tot(f(v),f(w))\in E(T_Q))).
  \end{align}
  For $i\in[\ell]$, let $(g_i,{\leq})\in\cS_{\ell,2}$ be the $\ell$-split order over $[2]$ given by $g_i(1) =
  g_i(2) = i$ and $1\leq 2$. Define the partial order $\preceq$ over $V(F)$ as
  \begin{align*}
    v\preceq w & \iff f(v) = f(w)\land ((v,w)\in R_<(F)\tot (g_{f(v)},{\leq})\in Q_<).
  \end{align*}
  It is clear that $(f,{\preceq})$ is an $\ell$-split order over $V(F)$.
  
  We claim that $(f,{\preceq})$ is an $E$-proper $Q$-split order of $F$. We know that $f$ is a proper coloring
  of the graph part of $F$, so we need to show that the order part of $F$ is $Q$-uniform with respect to
  $(f,{\preceq})$. But this follows from the definition of $\preceq$ and~\eqref{eq:fTQ}; this concludes the
  proof of our claim.

  From our claim, it follows that
  \begin{align*}
    \chi(I_\cF^<)
    & =
    \sup\{\ell\in\NN_+ \mid \forall F\in\cF, \ell < \chi_<(F)\}\cup\{0\} + 1
    \\
    & =
    \max\{\chi_<(\cF), 1\},
  \end{align*}
  as desired.
\end{proof}

Let us note that the result of~\cite{BKV03} that proves an analogue of Theorem~\ref{thm:ESSAS} in terms of the
cyclic interval chromatic number (which has the same definition as the interval chromatic number, but
intervals are considered in the cyclic order) can also be retrieved from Theorem~\ref{thm:Ramseynoninduced}
with a similar proof to that of Proposition~\ref{prop:ordered}.

Finally, the result of~\cite{GMNPTV19} that proves an analogue of Theorem~\ref{thm:ESSAS} in terms of the
edge-order chromatic number follows trivially from Theorem~\ref{thm:noninduced}.

\section{Conclusion}
\label{sec:conclusion}

One property of the usual chromatic number satisfies is principality: in the setting of
Proposition~\ref{prop:usual} (i.e., the setting of the original Theorem~\ref{thm:ESSAS}), we have
$\chi(I_\cF^+) = \max\{\chi(\cF),1\} = \min\{\chi(I_{\{F\}}^+) \mid F\in\cF\}\cup\{1\}$, that is, the chromatic
number corresponding to a non-empty family of graphs is simply the minimum of the chromatic numbers
corresponding to its elements.

In the more general setting of Theorem~\ref{thm:Ramseynoninduced}, let $\cL$ be a language and let $\cF_0$ be a
family of structures on $\cL\cup\{E\}$. For another family $\cF$ of structures on $\cL\cup\{E\}$, we let
$I_\cF\:\TGraph\leadsto\Forb_{\TGraph\cup T_\cL}((\cF_0\cup\cF)\up^E)$ act identically on $E$. We say that
$T\df\Forb_{\TGraph\cup T_\cL}(\cF_0\up^E)$ satisfies the \emph{principality property} if
\begin{align*}
  \chi(I_\cF) & = \min\{\chi(I_{\{F\}}) \mid F\in\cF\}
\end{align*}
for every non-empty $\cF$.

The setting of Proposition~\ref{prop:usual} shows that $\TGraph$ satisfies the principality
property. Proposition~\ref{prop:ordered} shows that $\TGraph\cup\TLinOrder$ satisfies the principality
property as well. Since an analogous result to Proposition~\ref{prop:ordered} holds for cyclically ordered
graphs (see~\cite{BKV03}) in terms of the cyclic interval chromatic number, it follows that the theory of
cyclically ordered graphs $\TGraph\cup\TCycOrder$ also satisfies the principality property. However, it was
observed in~\cite{GMNPTV19} that the theory of edge-ordered graphs does not satisfy the principality
property. A natural question then is what theories satisfy the principality property?

\medskip

Let us also note that just as Theorem~\ref{thm:ESSAS}, Theorem~\ref{thm:abstractchromatic} also fails to
completely characterize the asymptotic behavior of the maximum number of copies of $K_t$ in $I(M)$ for
$M\in\cM[T]$ when $\chi(I)\leq t$. Even for the case $t=2$, the study of this problem when $\chi(I)\leq 2$ has
been done in a case by case manner and we refer the interested reader again to~\cite{PT06, BKV03, GMNPTV19,
  Tar19} for some of these results for graphs with extra structure.

In Section~\ref{sec:Ramsey} we proved the finiteness of the partite Ramsey numbers, but we made no attempt at
optimizing the upper bounds that can be derived from its proof. Just as with the classical Ramsey numbers,
providing good upper bounds is a very interesting problem in its own right and some work has been done in the
non-partite case for some specific theories~\cite{NR89,CS16,CFLS17,CS18,BV20}.

Let us also point out that the partite Ramsey numbers that we studied can be viewed as the diagonal case. The
non-diagonal case can be defined as follows: given a function $h\:\cP_{\ell,\cL}\to\NN^\ell$ and
$\vec{n}\df(n_1,\ldots,n_\ell)\in\NN^\ell$, we write $\vec{n}\xrightarrow{T} h$ if for every model $M$ of $T$
and every $\ell$-split order $(f,{\preceq})\in\cS_{\ell,V(M)}$ on $V(M)$ with $\lvert f^{-1}(i)\rvert\geq n_i$
for all $i\in[\ell]$, there exists an $\ell$-Ramsey pattern $Q\in\cP_{\ell,\cL}$ over $\cL$ and a set
$W\subseteq V(M)$ such that $\lvert f^{-1}(i)\cap W\rvert\geq h(Q)_i$ for all $i\in[\ell]$ and $M\rest_W$ is
$Q$-uniform with respect to $(f\rest_W, {\preceq_W})$. It follows that for $m\df\max\{h(Q)_i \mid
Q\in\cP_{\ell,\cL}\land i\in[\ell]\}$, if $\min\{n_i \mid i\in[\ell]\}\geq R_{\ell,T}(m)$, then
$\vec{n}\xrightarrow{T} h$. Just as in the classical Ramsey theory, studying the off-diagonal case is an
interesting problem as well.

\bibliographystyle{alpha}
\bibliography{refs}

\end{document}